\documentclass{amsart}
\usepackage{amsmath,amssymb,amsthm,amsfonts,amscd,mathrsfs,tikz-cd}
\usepackage[all,cmtip]{xy}
\usepackage{graphicx}
\usepackage{subcaption}

\theoremstyle{plain}

\newtheorem{theorem}{Theorem}[section]
\newtheorem{proposition}[theorem]{Proposition}
\newtheorem{lemma}[theorem]{Lemma}
\newtheorem{corollary}[theorem]{Corollary}

\newtheorem*{proposition*}{Proposition}

\theoremstyle{definition}
\newtheorem{definition}[theorem]{Definition}
\newtheorem{example}[theorem]{Example}

\theoremstyle{remark}
\newtheorem{remark}[theorem]{Remark}

\newcommand{\secref}[1]{Section~\ref{#1}}
\newcommand{\thmref}[1]{Theorem~\ref{#1}}

\newcommand{\lemref}[1]{Lemma~\ref{#1}}
\newcommand{\corref}[1]{Corollary~\ref{#1}}

\newcommand{\exref}[1]{Example~\ref{#1}}

\newcommand{\remref}[1]{Remark~\ref{#1}}

\newcommand{\defref}[1]{Definition~\ref{#1}}
\newcommand{\figref}[1]{Figure~\ref{#1}}

\numberwithin{subsection}{theorem}
\newcommand{\R}{\mathbb{R}}
\newcommand{\Z}{\mathbb{Z}}

\begin{document}

\title[Subdivision and Digital Images]{Subdivision of Maps of Digital Images}

\author{Gregory Lupton}
\author{John Oprea}
\author{Nicholas A. Scoville}

\address{Department of Mathematics, Cleveland State University, Cleveland OH 44115 U.S.A.}

\email{g.lupton@csuohio.edu}
\email{j.oprea@csuohio.edu}

\address{Department of Mathematics and Computer Science, Ursinus College, Collegeville PA 19426 U.S.A.}

\email{nscoville@ursinus.edu}

\date{\today}


\begin{abstract}   With a view towards providing tools for analyzing and understanding digitized images, various notions from algebraic topology have been introduced into the setting of digital topology.  In the ordinary topological setting, invariants such as the fundamental group are invariants of homotopy type.  In the digital setting, however, the usual notion of homotopy leads to a very rigid invariance that does not correspond well with the topological notion of homotopy invariance.  In this paper, we establish fundamental results about subdivision of maps of digital images with $1$- or $2$-dimensional domains.  Our results lay the groundwork for showing that the digital fundamental group is an invariant of a much less rigid equivalence relation on digital images, that is more akin to the topological notion of homotopy invariance.  Our results also lay the groundwork for defining other invariants of digital images in a way that makes them invariants of this less rigid equivalence.
\end{abstract}

\thanks{This work was partially supported by grants from the Simons Foundation: (\#209575 to Gregory Lupton
and \#244393 to John Oprea).}

\maketitle

\section{Introduction}

In digital topology, the basic object of interest  is a  \emph{digital image}: a finite set of integer lattice points in an ambient Euclidean space with a suitable adjacency relation between points.  This is an abstraction of an actual digital image which consists of  pixels (in the plane, or higher dimensional analogues of such).

There is an extensive literature with many results that use ideas from topology in this setting (e.g. \cite{Ro86, Bo99, Evako2006}).
In many instances, however, notions from topology have been translated directly into the digital setting in a way that results in digital versions of topological notions  that  are very rigid   and hence have limited applicability.
 In contrast to this existing literature, in \cite{LOS19a} we have started to  build a more general ``digital homotopy theory" that
 brings the full strength of homotopy theory to the digital setting.  In our approach, we aim to use less rigid constructions, with a view towards broad applicability and greater depth of development.
 A key ingredient in such an approach is \emph{subdivision}.  However, the behaviour of maps with respect to subdivision is not well-understood.
In this paper, we  establish  fundamental results about subdivision of maps of digital images with $1$-dimensional (1D) and $2$-dimensional (2D) domains.  The utility of our results is indicated in \cite{LOS19c}, in which we define a digital fundamental group and show that it is an invariant of subdivision-homotopy equivalence, which is a concept of ``sameness" for spaces that is much less rigid than the notion of homotopy equivalence that is commonly used in digital topology.  Our results of \cite{LOS19a, LOS19c}, both in the basic constructions and in the developments, emphasize subdivision as a basic feature, whereas in those of \cite{Bo99} and many other articles in the digital topology literature, subdivision plays a background role at most.   Our results here on subdivision of maps also allow us to define invariants of 2D digital images such as Lusternik-Schnirelmann category in a way that is much less rigid than previously done (e.g.~as in \cite{Bo-Ve18}).  In general, our results work towards establishing  ``subdivision versions" of the usual invariants.  Our motivating point of view is that one should incorporate subdivision at a basic level, rather than directly translate a definition or construction from the topological to the digital setting.  Incorporating subdivision results in digital invariants whose behaviour  more closely follows that of their topological counterparts, when compared to the commonly used digital invariants that do not incorporate subdivision.   To do this generally, however, requires a fuller understanding of the behaviour of maps with respect to subdivision---maps with domains of arbitrary dimension.

The paper is organized as follows.  In \secref{sec: basics} we review standard definitions and terminology, and set our conventions (especially with regard to adjacency).    In \secref{sec: subdivision} , we give a thorough discussion of subdivision of digital images and maps of digital images.  We show how subdivision may be broken down into a succession of partial subdivisions (\corref{cor: factor rho}).  Several figures are included that serve to indicate the basic ideas and concerns.  The main question, illustrated through examples,  is how---or even whether---a map of digital images induces one on subdivisions.  In \secref{sec: 1D}, we resolve this question for maps of digital images whose domain is an interval, namely paths and loops in a digital image (of any dimension). In \secref{sec: 2D} we do likewise for  maps whose domain is a 2D digital image.  In each case, we construct a canonical map of subdivisions from a given map of digital images.  The main results are \thmref{thm: path odd subdivision map} and  \thmref{thm: 2-D subdivision map}.  A brief indication of the way in which our results here may be applied is given in \secref{sec: homotopy}.  But applications of and developments from  these results appear elsewhere.
There, we also indicate how our results here on subdivision of maps lay the groundwork for future developments.

\section{Basic Notions: Adjacency, Continuity, Products}\label{sec: basics}

In this paper, a \emph{digital image (of dimension $n$)} $X$ means a finite subset $X \subseteq \Z^n$ of the integral lattice in some $n$-dimensional Euclidean space, together with
a particular  adjacency relation inherited from that of $\Z^n$.  Namely, two points $x = (x_1, \ldots, x_n) \in \Z^n$ and  $y = (y_1, \ldots, y_n) \in \Z^n$  are adjacent if their coordinates satisfy $|x_i-y_i| \leq 1$ for each $i = 1, \dots, n$.

\begin{remark}
In the literature, it is common to allow for various choices of adjacency.  For example, a planar digital image is a subset of $\Z^2$ with either ``$4$-adjacency" or ``$8$-adjacency" (see, e.g.~Section 2 of \cite{Bo99}).  However, in this paper, we always assume (a subset of) $\Z^n$ has the highest degree of adjacency possible ($8$-adjacency in $\Z^2$, $26$-adjacency in $\Z^3$, etc.).   In fact, there is a philosophical reason for our fixed choice of adjacency relation:  It is effectively forced on us by the considerations of  \defref{def: products} and \exref{ex: diagonal} below.
\end{remark}

If $x, y \in X \subseteq \Z^n$, we write $x \sim_X y$ to denote that $x$ and $y$ are adjacent in $X$. For digital images $X\subseteq \Z^n$ and $Y \subseteq \Z^m$, a  function $f\colon X \to Y$  is \emph{continuous} if $f(x) \sim_Y f(y)$ whenever $x \sim_Xy$.
By a \emph{map} of digital images, we mean a continuous function.  Occasionally, we may encounter a non-continuous function of digital images.  But, mostly, we deal with maps---continuous functions---of digital images.  The \emph{composition of maps} $f\colon X \to Y$ and $g\colon Y \to Z$ gives a (continuous) map $g\circ f\colon X \to Z$, as is easily checked from the definitions.

An \emph{isomorphism} of digital images is a continuous bijection $f \colon X \to Y$ that admits a continuous inverse $g \colon Y \to X$, so that  we have $f\circ g = \mathrm{id}_Y$ and $g\circ f = \mathrm{id}_X$, and $g$ is also bijective.  If $f \colon X \to Y$ is an isomorphism of digital images, then we say that $X$ and $Y$ are isomorphic digital images, and write $X \cong Y$.

\begin{example}\label{ex:basic digital images}
We use the notation $I_N$ for the \emph{digital interval of length} $N$, namely $I_N \subseteq \Z$ consists of the integers from $0$ to $N$ in $\Z$, and consecutive
integers are adjacent. Thus, we have $I_1 = [0, 1] = \{0, 1\}$, $I_2 = [0, 2] = \{0, 1, 2\}$, and so-on.  Occasionally, we may use $I_0$ to denote the singleton point $\{0\} \subseteq \Z$. As an example in $\Z^2$, consider what we call  \emph{the Diamond}, $D = \{ (1, 0), (0, 1), (-1, 0), (0, -1) \}$, which may be viewed as a digital circle.  Note that pairs of vertices all of whose coordinates differ by $1$, such as  $(1, 0)$ and  $(0, 1)$ here, are adjacent according to our definition.  Otherwise, $D$ would be disconnected.
\begin{figure}[h!]
\centering
   \begin{subfigure}{0.49\linewidth} \centering
    \includegraphics[trim=160 350 80 80,clip,width=\textwidth]{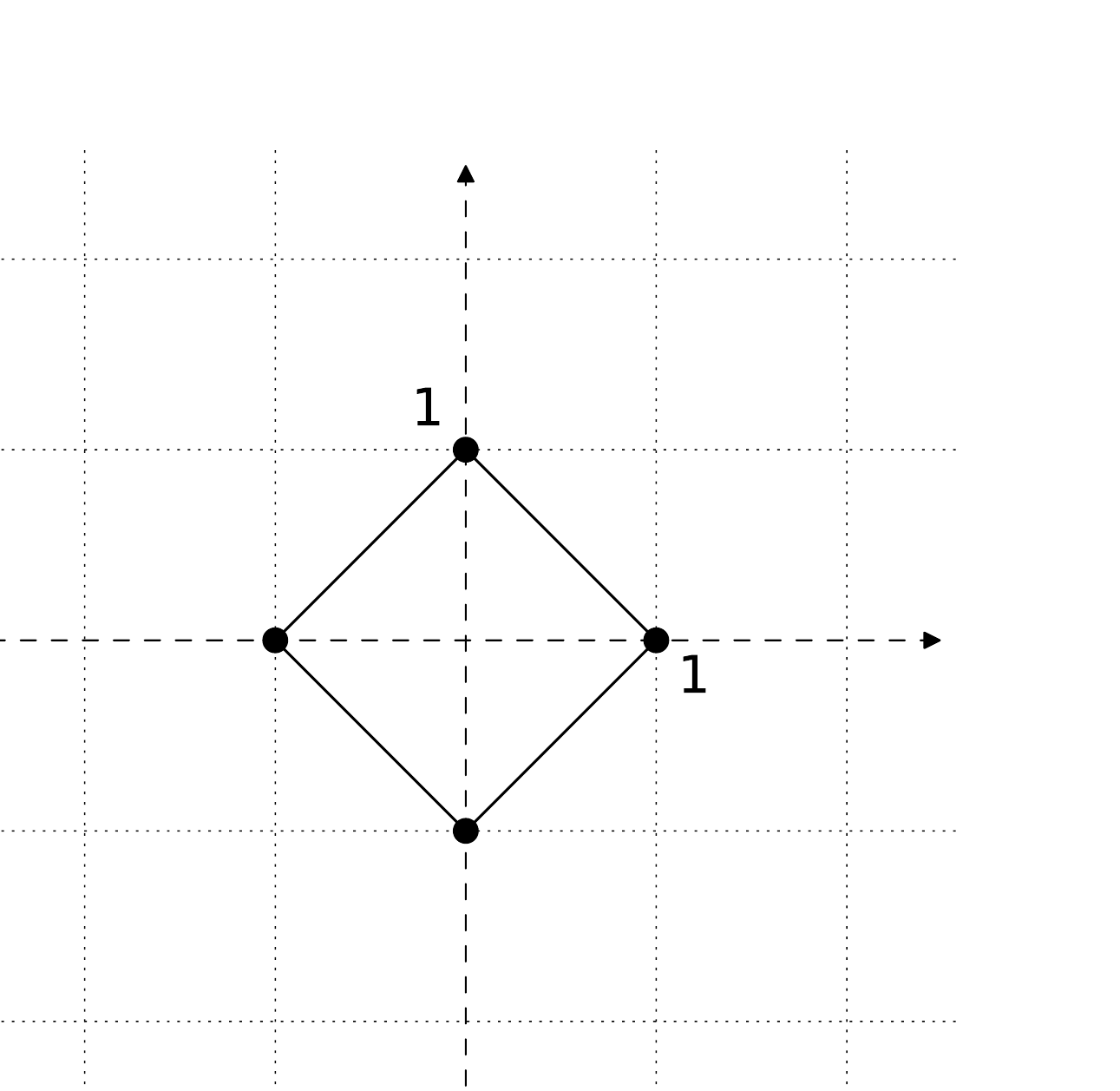}
     \caption{$D$: The Diamond}\label{fig:Circle D}
   \end{subfigure}
   \begin{subfigure}{0.49\linewidth} \centering
    \includegraphics[trim=160 350 80 80,clip,width=\textwidth]{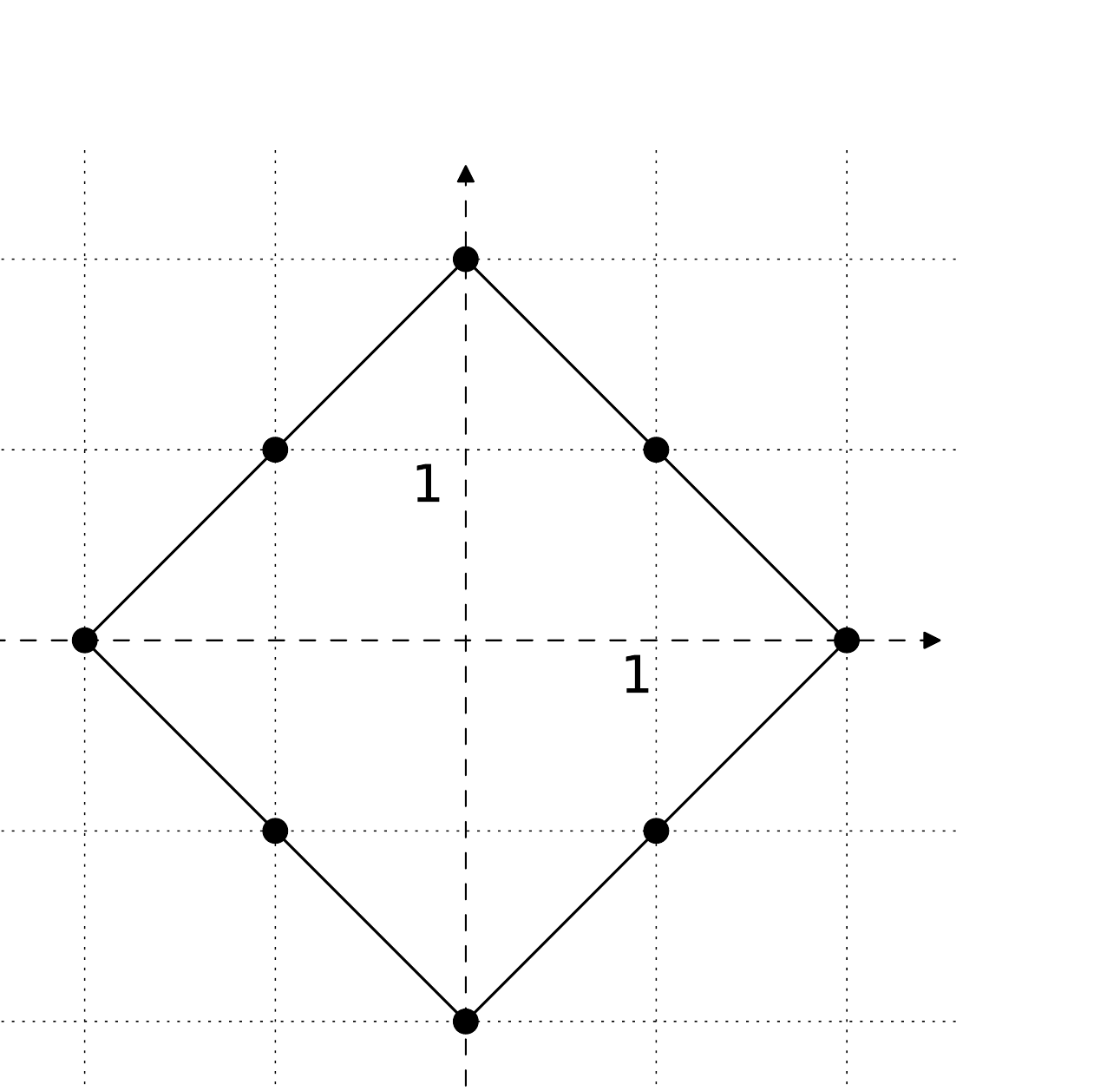}
     \caption{$C$: A larger digital circle}\label{fig:Circle C}
   \end{subfigure}
\caption{Two digital circles.} \label{fig:D & C}
\end{figure}
In \figref{fig:D & C} we have included the axes (dashed) and also indicated adjacencies (solid) in the style of a graph.  Note, though, that we have no choice as to which points are adjacent: this is determined by position, or coordinates, and we do not choose to add or remove edges here.
 As an example in  $\Z^3$, we have $S = \{ (1, 0, 0), (0, 1, 0), (-1, 0, 0), (0, -1, 0), (0, 0, 1), (0, 0, -1) \}$ (the vertices of an octahedron, with adjacencies corresponding to the edges of the octahedron).  This may be viewed as a digital $2$-sphere, and the pattern emerging here may be continued to a digital $n$-sphere in $\Z^{n+1}$ with $2n+2$ vertices. The map $f\colon I_2 \to I_1$ given by $f(0) = 0$, $f(1) = 0$, and $f(2) = 1$ is continuous, but the function  $g\colon I_1 \to I_2$ given by $g(0) = 0$, $g(1) = 2$ is not: we cannot ``stretch" an interval to a longer one.  Likewise, suppose we enlarge $D$ to the bigger digital circle   $C = \{ (2, 0), (1, 1), (0, 2), (-1, 1), (-2, 0), (-1, -1), (0, -2), (1, -1) \}$ (see \figref{fig:D & C}).  Then the only maps $D \to C$ will be ``homotopically trivial:" we cannot ``wrap" a smaller circle around a larger one.
\end{example}

The last comment of the preceding example points to the main motivation for the results of this paper. Whereas homotopy is not the main focus of this paper (the notion is reviewed here in \secref{sec: homotopy}),  our results here are motivated by wanting to relax the notion of homotopy equivalence commonly used in digital topology.  We can give the basic idea informally, as follows. Because we cannot ``wrap" a smaller circle around a larger one, digital circles of different sizes are not homotopy equivalent, in the sense commonly used in digital topology.  But from a (topological) homotopy point of view, it seems reasonable to view $D$ and $C$ as above---more generally, digital circles of different sizes---as being equivalent.  In \cite{LOS19a, LOS19c}, we develop a notion of \emph{subdivision-homotopy equivalence} of digital images, which is a notion of ``sameness" of digital images that combines subdivision with homotopy equivalence, and which is a less rigid notion of ``sameness" than digital homotopy equivalence.  Indeed, it turns out that $D$ and $C$ are subdivision-homotopy equivalent, but not homotopy equivalent.  The comments made here about $D$ and $C$ are discussed in detail in Exercise 3.22 of  \cite{LOS19c}.

\begin{definition}[digital products]\label{def: products}
The product of digital images $X\subseteq \Z^m$ and $Y\subseteq \Z^n$ is the Cartesian product of sets $X \times Y \subseteq \Z^{m} \times \Z^{n} \cong \Z^{m+n}$ with the adjacency relation $(x, y) \sim_{X \times Y} (x', y')$ when $x\sim_X x'$  and $y \sim_Y y'$.
\end{definition}

In fact,  this is tantamount to our assumption that $\Z^n$, and any digital image in it, has the highest degree of adjacency possible, with the isomorphisms $\Z^n \cong \Z^r \times \Z^{n-r}$ for $r = 1, \ldots, n-1$.  Note that some authors in the literature use a different adjacency relation on the product: the \emph{graph product}, whereby $(x, y)$ is adjacent to $(x', y')$ if $x = x'$ and $y\sim_Y y'$, or $x \sim_X x'$ and $y = y'$.  The notion we use is sometimes called the  \emph{strong product}, in a graph theory setting.  Our definition of (adjacency on) the product means that it is the categorical product, in the category of (finite) digital images and digitally continuous maps.  This point is explained in the following statement.

\begin{lemma}\label{lem: product}
For digital images $X\subseteq \Z^m$ and $Y\subseteq \Z^n$, the projections onto either factor $p_1\colon X \times Y \to X$ and $p_2\colon X \times Y \to Y$ are continuous.  Suppose given maps of digital images $f\colon A \to X$ and $g\colon A \to Y$.  Then there is a unique map, which we write $(f, g)\colon A \to X \times Y$ that satisfies $p_1\circ (f, g) = f$ and $p_2\circ (f, g) = g$.
\end{lemma}

\begin{proof}
The first assertion follows immediately from the definitions.   The map $(f, g)$ is defined as $(f, g)(a) = \big( f(a), g(a) \big)$.  It is immediate from the definitions that this map is continuous.  This is evidently the unique map with the suitable coordinate functions.
\end{proof}

\begin{example}\label{ex: diagonal}
For $X \subseteq \Z^n$ a digital image, the \emph{diagonal map}
$$\Delta\colon X \to X \times X \subseteq \Z^n \times \Z^n \cong \Z^{2n}$$
 is defined as $\Delta(x) = (x, x)$ for each $x \in X$.   Suppose we have $X = I_1 \subseteq \Z$, with  $\Delta\colon I_1 \to I_1 \times I_1$.  Since $0 \sim_X 1$, we need $(0,0) \sim_{X \times X} (1, 1)$ if the diagonal is to be continuous, which of course we do have with our conventions.
\end{example}

Because of the rectangular nature of the digital setting, it is often convenient to consider the \emph{product of maps}, as follows.

\begin{definition}\label{def: map product}
Given functions of digital images $f_i \colon X_i \to Y_i$ for $i = 1, \ldots, n$, we define the \emph{product} function
$$f_1 \times \cdots \times f_n \colon X_1 \times \cdots \times  X_n \to Y_1 \times \cdots \times Y_n$$
as $(f_1 \times \cdots \times f_n) (x_1, \ldots, x_n) = \big(f_1(x_1), \ldots, f_n(x_n) \big)$.
\end{definition}

\begin{lemma}\label{lem: map product conts}
Given continuous maps of digital images $f_i \colon X_i \to Y_i$ for $i = 1, \ldots, n$, their product $f_1 \times \cdots \times f_n$ is a (continuous) map.
\end{lemma}

\begin{proof}
This follows directly from the definitions.
\end{proof}

We will make use of the product of maps towards the end of the following section and in the sequel.  This gives another reason for why we want the product of digital images to be defined as in \defref{def: products}.

\section{Subdivision}\label{sec: subdivision}

The notion of \emph{subdivision of a digital image} plays a fundamental role in our development of ideas in the digital setting, and  is a main focus of this paper.

\begin{definition}\label{def: subdivision}
Suppose that $X\subseteq \Z^n$ is an $n$-dimensional digital image.  For each $k \geq 2$, we have the $k$-\emph{subdivision of $X$}, which is an auxiliary (to $X$) $n$-dimensional digital image  denoted by $S(X, k)\subseteq \Z^n$, together with a canonical map or \emph{standard projection}
$$\rho_k\colon  S(X, k) \to X$$
that is continuous in our digital sense.  For a real number $x$, denote by $\lfloor x \rfloor$ the greatest integer less-than-or-equal-to $x$ (the \emph{integer floor of $x$}).  First, make
the $\Z[1/k]$-lattice in $\R^n$, namely, those points with coordinates each of which is $z/k$ for some integer $z$, and then set
$$S'(X, k) = \left\{ (x_1, \ldots, x_n) \in \left(\Z\left[\frac{1}{k}\right]\right)^n  \mid ( \lfloor x_1 \rfloor, \ldots, \lfloor x_n \rfloor) \in X \right\}.$$
Then set
$$ S(X,k) = \left\{ (kx_1, \ldots, kx_n) \in \Z^n  \mid (x_1, \ldots, x_n) \in S'(X, k) \right\}.$$
The map $\rho_k$ is given by $\rho_k\big( (y_1, \ldots, y_n) \big) = ( \lfloor y_1/k \rfloor, \ldots, \lfloor y_n/k \rfloor)$, and one checks that this map is continuous.

For $x \in X$ an individual point, we write $S(x, k) \subseteq S(X, k)$ for the points $y \in S(X, k)$ that satisfy $\rho_k(y) = x$.  If $x = (x_1, \ldots, x_n)$ is a point in an $n$-dimensional digital image, then we may describe this set in general as
\begin{equation}\label{eq: point subdn}
S(x, k) = \{ (kx_1 + r_1, \ldots, kx_n + r_n) | 0 \leq r_i \leq k-1 \}.
\end{equation}
That is, for each $x \in X$, $S(x, k)$ is an $n$-dimensional cubical lattice in $\Z^n$ with each side of the cubical lattice containing $k$ points.  Notice that  the result of subdivision therefore depends on the ambient space of the digital image.
\end{definition}

Occasionally, it may be convenient to extend \defref{def: subdivision} to include $k=1$, in which case we use the notational convention that $S(X, 1) = X$, and $\rho_1\colon  S(X, 1) \to X$ is just the identity map of $X$.

\begin{example}\label{ex: subdn}
Generally, subdivision of an interval $I_N \subseteq \Z$ gives a longer interval: We have $S(I_N, k) = I_{Nk+k-1} \subseteq \Z$.
Suppose that we have $X = I_2 = [0, 2] \subseteq \Z^2$.  Then we have $S(X, 2) = I_5 =   [0,  5] \subseteq \Z$, and $\rho_2\colon S(X, 2) \to X$ is given by $\rho_2(0) = \rho_2(1) = 0$, $\rho_2(2) = \rho_2(3) = 1$, and $\rho_2(4) = \rho_2(5) = 2$.   In \figref{fig:subd interval}, we indicate the way in which, for the same interval $I_2$,  the projection $\rho_3 \colon S(I_2, 3) \to I_2$ aggregates points in the subdivided interval to map them back to the original.
\begin{figure}[h!]
\centering
\includegraphics[trim=130 620 100 120,clip,width=\textwidth]{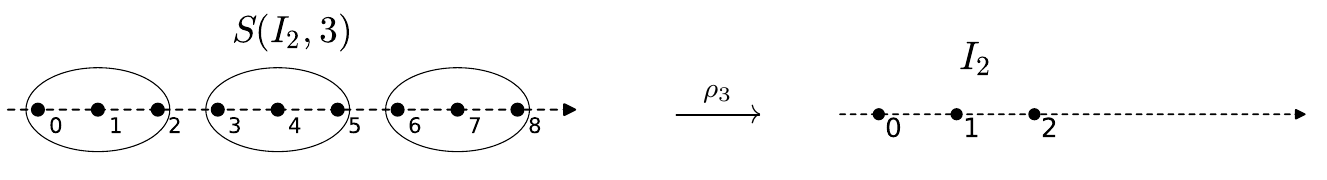}
\caption{\label{fig:subd interval} Aggregation of points by $\rho_3\colon S(I_2, 3) = I_8 \to I_2$.}
\end{figure}
We also note here that $S(I_0, k) = S(\{0\}, k) = I_{k-1}$

As a two-dimensional example, suppose that we have $X = \{ (0,0), (1, 1) \}  \subseteq \Z$.  Then $S(X, 2) = \{ (0,0), (1, 0), (0,1), (1, 1), (2,2), (2, 3), (3,2), (3, 3) \}$, and we have
$\rho_2\colon S(X, 2) \to X$  given by $\rho(0,0) = \rho(1, 0) =  \rho(0,1) = \rho(1, 1)  = (0, 0)$, and $\rho(2,2) = \rho(2, 3) =  \rho(3,2) = \rho(3, 3)  = (1, 1)$.  Finally, in \figref{fig:subd diamond}, we show the points of $S(D, 2)$, with $D$ the diamond as in \figref{fig:D & C} above, and indicate the way in which the points of $S(D, 2)$ are aggregated by the projection $\rho_2 \colon S(D, 2) \to D$.
\begin{figure}[h!]
\centering
\includegraphics[trim=130 520 120 110,clip,width=\textwidth]{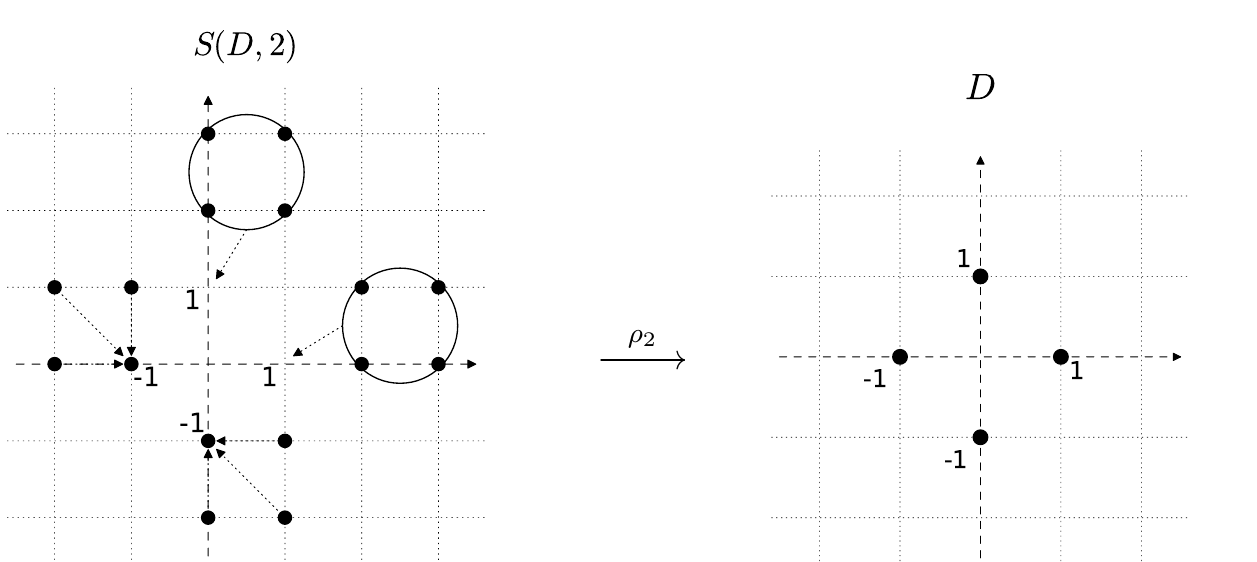}
\caption{\label{fig:subd diamond} Aggregation of points by $\rho_2\colon S(D, 2) \to D$.}
\end{figure}
\end{example}

Subdivision behaves well with respect to products.  For any digital images $X \subseteq \Z^m$ and $Y \subseteq \Z^n$ and any $k \geq 2$ we have an isomorphism of digital images
$$S(X \times Y, k) \cong S(X, k) \times S(Y, k)$$
and, furthermore, the standard projection $\rho_k\colon S(X \times Y, k) \to X \times Y$ may be identified with the product of the standard projections on $X$ and $Y$, thus:
$$\rho_k = \rho_k \times \rho_k\colon  S(X, k) \times S(Y, k) \to X \times Y.$$
Note also that we may iterate subdivision.  It is straightforward to check that, for any $k, l \geq 1$, we have an isomorphism of digital images
$S\big( S(X, k), l\big) \cong S(X, kl)$.

\begin{example}\label{ex: product subdns}
We mentioned above that, for $I_0 = \{0\} \subseteq \Z$, we have  $S(I_0, k) = S(0, k) = I_{k-1}$.  For the origin $\mathbf{0} = (0, \ldots, 0) \in \Z^n$, we have $S(\mathbf{0}, k) = (I_{k-1})^n$, an $n$-cube in $\Z^n$, and we may identify the projection $\rho_k\colon (I_{k-1})^n \to \{ \mathbf{0}\}$ as a product of projections
$$\rho_k \times \cdots \times \rho_k \colon I_{k-1} \times \cdots \times I_{k-1} \to I_0  \times \cdots \times I_0.$$
More generally, for any $x \in \Z$, we have
$$S(x, k) = [kx, kx+k-1] = \{ kx + r \mid 0 \leq r \leq k-1 \}.$$
If $x = (x_1, \ldots, x_r) \in \Z^n$, then we have
$$
\begin{aligned}
S(x, k) &=  \{ (kx_1 + r_1, \ldots, kx_n + r_n )\mid 0 \leq r_i \leq k-1 \} \\
&= [kx_1, kx_1+k-1] \times \cdots \times [kx_n, kx_n+k-1]\\
&= S(x_1, k) \times \cdots \times S(x_n, k).
\end{aligned}
$$
These descriptions make plain that we may identify the projection $\rho_k\colon S(x, k) \to \{ x\}$ with the product of projections
$$\rho_k \times \cdots \times \rho_k \colon S(x_1, k) \times \cdots \times S(x_n, k) \to \{ x_1\}  \times \cdots \times \{ x_1\}.$$
\end{example}

By an \emph{inclusion of digital images} (of the same dimension) $j \colon A \to X \subseteq \Z^n$ we mean that $A$ is a subset of $X$ (the coordinates of a point of $A$ remain the same under inclusion into $X$).
It is easy to see that, given an inclusion of digital images  $j \colon A \to X \subseteq \Z^n$, we have an obvious corresponding continuous  inclusion of subdivisions $S(j, k)\colon S(A, k) \to S(X, k)$ such that the diagram
$$\xymatrix{
S(A, k) \ar[d]_{\rho_k} \ar[r]^{S(j, k)} & S(X, k) \ar[d]^{\rho_k}\\
A \ar[r]_{j} & X}$$
commutes.  We say that the map  $S(j, k)$ \emph{covers} the map $j$.   Indeed, we may give an explicit formula as follows.  For each point  $a \in A$, write $a = (a_1, \ldots, a_n)$. Also, write $t = (t_1, \ldots, t_n)$, with $0 \leq t_1, \ldots, t_n \leq k-1$, for a typical point $t$ in the cubical $k \times k \times \cdots \times k$ lattice $(I_{k-1})^n\subseteq \Z^n$.  Then the points of $S(a, k) \subseteq S(A, k)$ may be written as
$$S(a, k) = \{ k \,a + t \mid t \in (I_{k-1})^n \} = \{ (ka_1 + t_1, \ldots, ka_n+t_n) \mid 0 \leq t_1, \ldots, t_n \leq k-1 \},$$
with $\rho_k( k\, a + t) = a$ for all $t \in (I_{k-1})^n$.  Here, the scalar multiple $k\,a$ and the  sum $k\,a + t$ denote coordinate-wise (vector) scalar multiplication and addition in $\Z^n$.   Then $S(j, k)\colon S(A, k) \to S(X, k)$ may be written as
\begin{equation}\label{eq: canonical subdn map}
S(j, k)\big(  k\,a + t  \big) =   k\,j(a) + t,
\end{equation}
where $j(a) = (a_1, \ldots, a_n) \in X$.  It is easy to confirm that this gives a (continuous) map.

For a more general map $f \colon X \to Y$, however, it is not so clear how we should construct a map of subdivisions that covers the map, in the sense of a filler---a map that occupies the place of the dotted arrow---for the following (commutative) diagram:
$$\xymatrix{
S(X, k) \ar[d]_{\rho_k} \ar@{.>}[r] & S(Y, k) \ar[d]^{\rho_k}\\
X \ar[r]_{f} & Y}$$
In fact, it is not even obvious that such a map of subdivisions always exists, in general.  In this paper we show that such a map does exist for arbitrary maps of digital images with 1D and 2D domains.  However, as the next several examples illustrate, the formulation of \eqref{eq: canonical subdn map} will not provide such a map in general.

\begin{example}\label{ex: no subdn interval}
Consider the constant map of 1D digital images $c\colon I_1 \to I_0=\{ 0 \}$, given by $c(1) = c(0) = 0$.  If we use the formulation of \eqref{eq: canonical subdn map} above to define a function
$$S(c, k)\colon S(I_1, k) \to S( I_0, k)$$
as $S(c, k)\big(  k\,a + t  \big) =   k\,c(a) + t$, then we have $S(c, k)( k-1) = S(c, k)(k\cdot 0 + k-1) = k\, c(0) + k-1 = k-1$ but  $S(c, k)( k)  = S(c, k)(k\cdot 1 + 0) = k\, c(1) + 0 =0$.  Then $k-1 \sim_{S(I_1, k)} k$ but $k-1 \not\sim_{S( I_0, k)} 0$ unless $k = 2$: the function $S(c, k)$ is not continuous for $k \geq 3$.  See \figref{fig:Ex2-7} for an illustration of this situation.
\begin{figure}[h!]
\centering
\includegraphics[trim=140 560 100 130,clip,width=\textwidth]{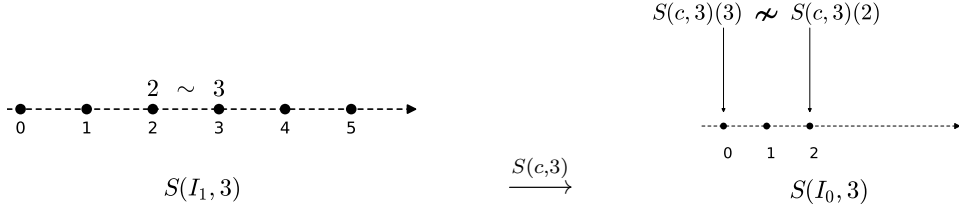}
\caption{\label{fig:Ex2-7} $S(c, k)\colon S(I_1, k) \to S( I_0, k)$ is not continuous ($k = 3$ pictured; we have $S(c, 3)(2) = 2$ but $S(c, 3)(3) = 0$).}
\end{figure}
In this example, defining $C \colon S(I_1, k) \to S( I_0, k)$ as a constant map, $C(k\,a + t) = 0$, for instance, gives a continuous map that covers $c$.  But the point here is, that it is not obvious how to adapt a covering map of subdivisions depending on the given map.
\end{example}

The issue is not confined to functions that coalesce points together, either.  Here are two examples of injective maps for which $S(f, k)$, defined as in \eqref{eq: canonical subdn map} above, fails to be continuous.

 \begin{example}\label{ex: no subdivision map}
 (a) Consider the map $f\colon I_1 \to I_1$ given by $f(0) = 1$ and $f(1) = 0$.  The function  $S(f, k)\colon S(I_1, k) = I_{2k-1} \to S(I-1, k)  = I_{2k-1}$ defined by the formulation of \eqref{eq: canonical subdn map} above  gives
$$S(f, k)(k-1) = 2k-1 \quad \text{and} \quad S(f, k)(k) = 0.$$
Since $2k-1 \not\sim 0$, this function is not continuous for any $k$.   See \figref{fig:Ex2-8a}, in which we have indicated the way in which the projections $\rho_k \colon S(I_1,k) \to I_1$ aggregate points.
\begin{figure}[h!]
\centering
\includegraphics[trim=145 470 20 115,clip,width=\textwidth]{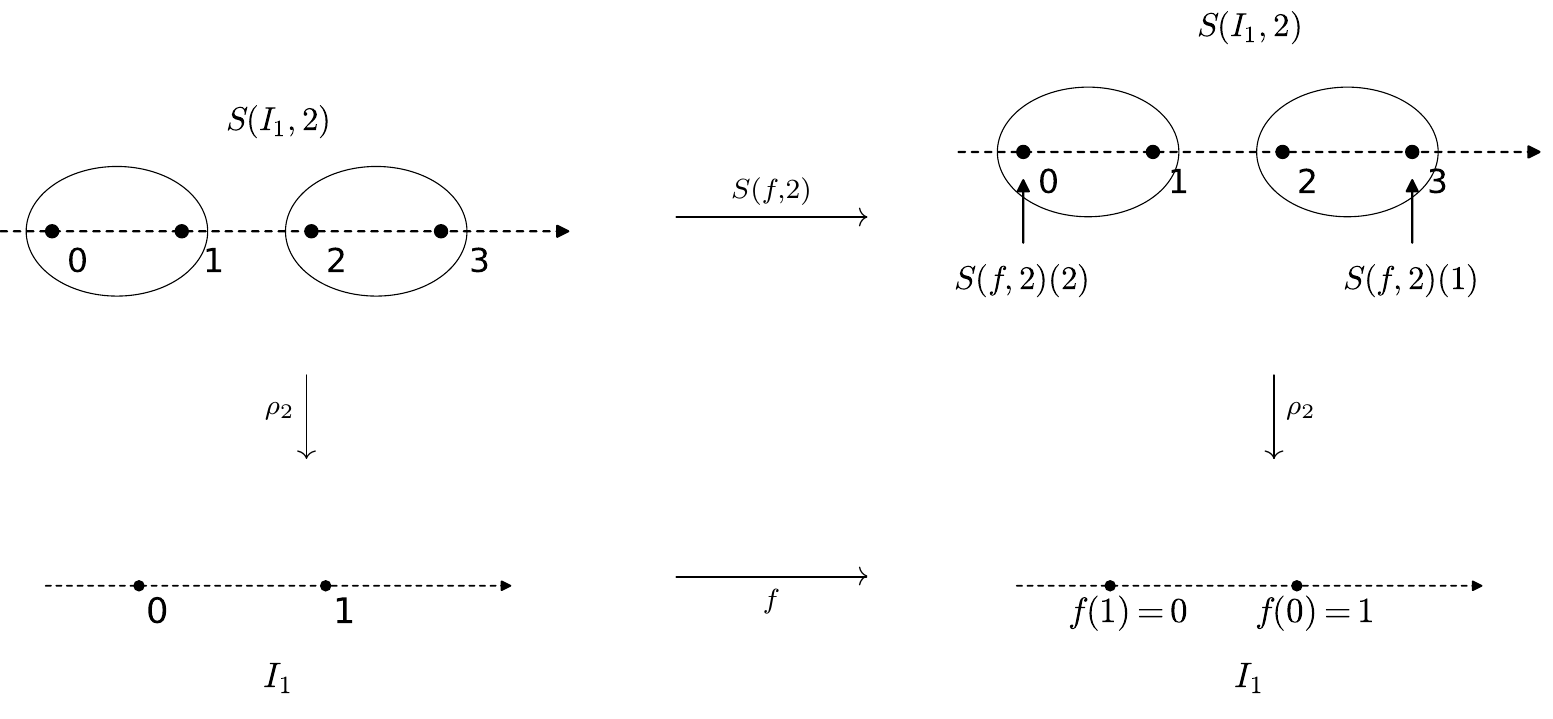}
\caption{\label{fig:Ex2-8a} $S(f, k)\colon S(I_1, k) \to S( I_1, k)$ is not continuous ($k = 2$ pictured).}
\end{figure}

 (b) (Similar to an observation illustrated in  \cite[Fig.1]{B-S16}.)  Consider the map $f \colon X \to Y$ with $X = \{(0, 0), (1, 0), (0, 1) \} \subseteq \Z^2$, $Y = \{(0, 0), (1, 1) \} \subseteq \Z^2$, and $f$ given by
 $$ f(0,0) = (0,0), \quad   f(1,0) = (1,1), \quad  f(0,1) = (1,1).$$
The function  $S(f, 2)\colon S(X, 2) \to S(Y, 2)$ defined by the formulation of \eqref{eq: canonical subdn map} above  gives
$$(0, 0) \mapsto (0, 0), \quad (1, 0) \mapsto (1, 0), \quad (0, 1) \mapsto (0, 1), \quad (1, 1) \mapsto (1, 1)$$
on the four points of $S\big( (0, 0), 2 \big)$.
Likewise for the four points in $S\big( (1, 0), 2 \big)$,  $S(f, 2)$ would give
$$(2, 0) \mapsto (2, 2), \quad (2, 1) \mapsto (2, 3), \quad (3, 0) \mapsto (3, 2), \quad (3, 1) \mapsto (3, 3).$$
\begin{figure}[h!]
\centering
\includegraphics[trim=140 350 50 120,clip,width=\textwidth]{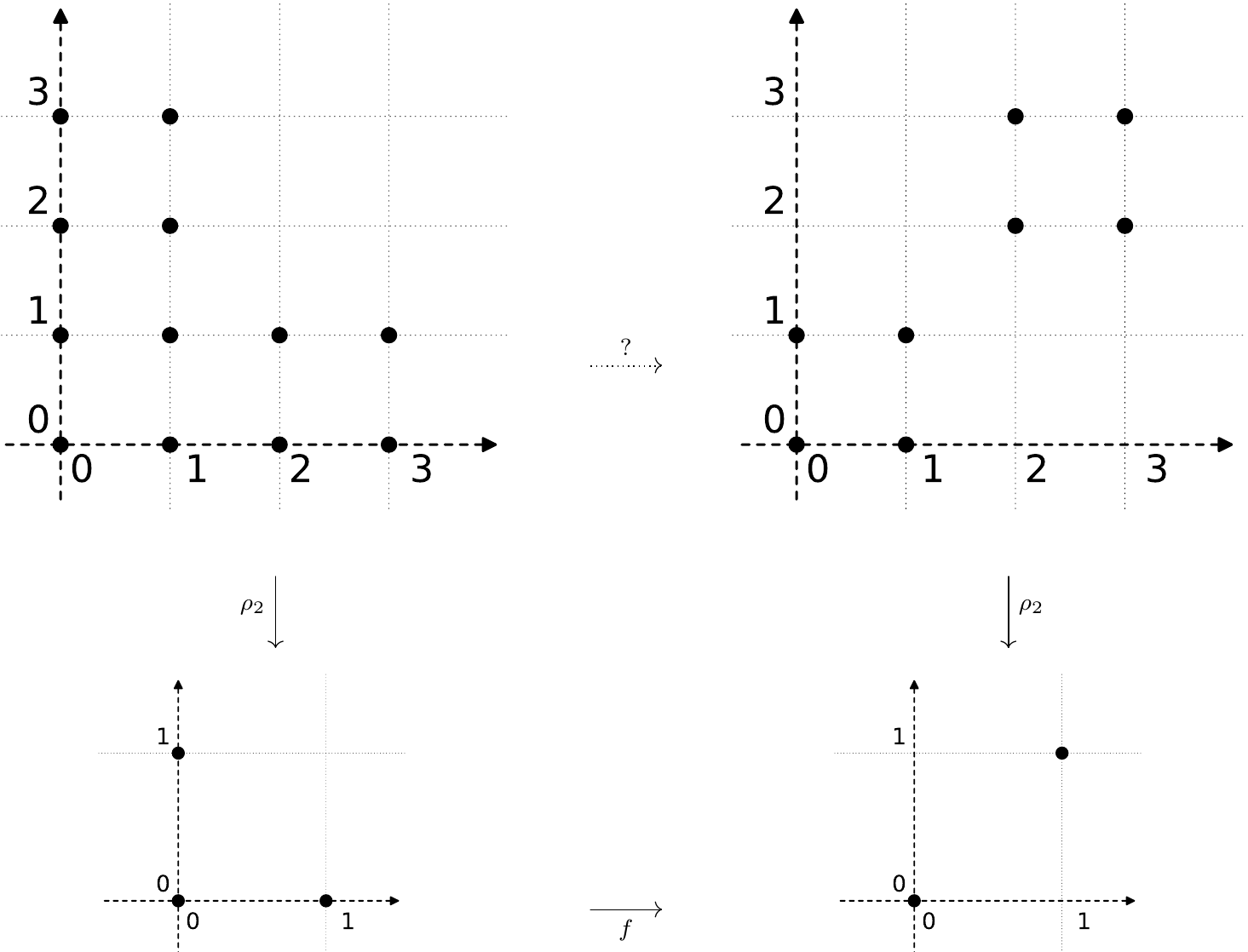}
\caption{\label{fig:nosubd} How to map subdivisions?}
\end{figure}
But this would result in adjacent points $(1, 0) \sim_{S(X, 2)} (2, 1)$ being mapped to non-adjacent points $(1, 0) \not\sim_{S(Y, 2)} (2, 3)$, for example.    The situation is summarized in  \figref{fig:nosubd}, in which we want a filler  $S(X, 2) \to S(Y, 2)$ that makes the diagram commute.  Notice one feature of this example, in particular.  Although we have $f(0, 0) = (0,0)$, it is not possible for a covering map of $f$ to restrict to the identity $S\big( (0, 0), k\big) \to S\big( (0, 0), k\big)$.  For $k=2$, for instance, we see in \figref{fig:nosubd}  that $(0, 1) \sim_X (0, 2)$ and  $(0, 1) \sim_X (1, 2)$,  but any covering map of $f$ must map both $(0, 2)$ and $(1, 2)$ to points of $S\big( (1, 1), 2\big)$ in $S(Y, 2)$, none of which are adjacent to $(0, 1) \in S(Y, 2)$. That is, the possibilities for a covering map are constrained by how surrounding points are mapped by $f$, and not just by how the points themselves are mapped.    In this example, it is not so clear how one should associate a continuous map $S(X, 2) \to S(Y, 2)$ to the original $f$, as part of a methodical scheme for doing so.
 \end{example}

 In the next two Sections, we will give methodical constructions that, in particular, provide covering maps of subdivisions in the examples above.  A more general question,  special cases of which are also resolved in the following sections, is to ask how---or whether---a map of digital images of different dimensions might induce a covering map of subdivisions.

We close this section on subdivision with some constructions that we use in the following section and in the sequel.  The projection $\rho_k \colon S(X, k) \to X$ may be factored---written as a composition---in various ways.  For example, if $k = pq$, then we may write
$$\rho_k = \rho_p\circ \rho_q \colon S(X, k) \to S(X, p) \to X.$$
A different sort of ``partial projection" that may also be used to factor $\rho_k$ is as follows.

\begin{definition}\label{def: rho^c}
For any $x \in \Z$ and any  $k \geq 2$, recall that the subdivision $S(x, k)$ may be described as $S(x, k) = [xk, xk + k-1]$.    Then, for $k \geq 3$, define a function
$$\rho^c_{k} \colon S(x, k) \to S(x, k-1)$$
as
$$\rho^c_{k}(xk + j) = \begin{cases}  xk + j & 0 \leq j \leq  \lfloor k/2 \rfloor -1\\
xk + j-1 &  \lfloor k/2 \rfloor \leq j \leq  k-1.\end{cases}
$$
Next, for any $x = (x_1, \ldots, x_n) \in \Z^n$, with the identifications from \exref{ex: product subdns} of
$$S(x, k) = S(x_1, k) \times \cdots \times S(x_n, k)$$
and
$$S(x, k-1) = S(x_1, k-1) \times \cdots \times S(x_n, k-1),$$
define $\rho^c_{k} \colon S(x, k) \to S(x, k-1)$ as the product of functions
$$\rho^c_{k}\times \cdots \times \rho^c_{k} \colon S(x_1, k) \times \cdots \times S(x_n, k) \to  S(x_1, k-1) \times \cdots \times S(x_n, k-1).$$
Finally, for any digital image $X \subseteq \Z^n$, define
$$\rho^c_{k} \colon S(X, k) \to S(X, k-1)$$
by viewing each subdivision as a (disjoint) union
$$ S(X, k) = \coprod_{x \in X}\ S(x, k) \qquad \text{and} \qquad S(X, k-1) = \coprod_{x \in X}\ S(x, k-1)$$
and assembling a global  $\rho^c_{k}$ on $S(X, k)$ from the individual $\rho^c_{k} \colon S(x, k) \to S(x, k-1)$ as just defined.
\end{definition}

\begin{proposition}\label{prop: rho^c conts}
For $k \geq 3$, the \emph{partial projection}  $\rho^c_{k} \colon S(X, k) \to S(X, k-1)$ is continuous.
\end{proposition}

\begin{proof}
For $x \in \Z$, the map of intervals $\rho^c_{k} \colon S(x, k) \to S(x, k-1)$ is easily seen to be continuous. Then, for any $x \in \Z^n$, we have defined
$\rho^c_{k} \colon S(x, k) \to S(x, k-1)$ as the product of individually continuous functions, hence it is also continuous.  It remains to confirm that the $\rho^c_{k}$ assemble together to give a globally continuous function on $S(X, k)$.

So suppose that we have $y  \in S(x, k)$ and $y' \in S(x', k)$ with $y \sim_{S(X, k)} y'$ and $x \not= x' \in X$.  Note, though,  that we must have $x \sim_X x'$, since $\rho_k(y) = x$,  $\rho_k(y') = x'$, and $\rho_k\colon S(X, k) \to X$ is continuous.   Write $x = (x_1, \ldots, x_n)$ and  $x' = (x'_1, \ldots, x'_n)$.  Then we have
$$y = (kx_1 + r_1, \ldots, kx_n + r_n) \qquad \text{and} \qquad y' = (kx'_1 + r'_1, \ldots, kx'_n + r'_n)$$
for $r_i, r'_i$ with $0 \leq r_i, r'_i \leq k-1$, each $i = 1, \ldots, n$.  Now for $\rho^c_{k}(y) \sim_{S(X, k-1)} \rho^c_{k}(y')$, it is necessary and sufficient that we have
$\rho^c_{k}(kx_i + r_i) \sim \rho^c_{k}(kx'_i + r'_i)$ in $S(x_i, k-1) \sqcup S(x'_i, k-1) \subseteq \Z$, for each $i$.  Write $\rho^c_{k}(kx_i + r_i) = (k-1)x_i + s_i$ and
$\rho^c_{k}(kx'_i + r'_i) = (k-1)x'_i + s'_i$, with the $s_i, s'_i$ satisfying  $0 \leq s_i, s'_i \leq k-2$ and determined as in \defref{def: rho^c}.  Then
$$\rho^c_{k}(kx_i + r_i) - \rho^c_{k}(kx'_i + r'_i) =   (k-1)(x_i - x'_i) + (s_i - s'_i),$$
and we must show that, for each $i = 1, \ldots, n$, we have
\begin{equation}\label{eq: rho^c difference}
-1 \leq  (k-1)(x_i - x'_i) + (s_i - s'_i) \leq 1.
\end{equation}
Because we have $x \sim_X x'$, it follows that, for each $i = 1, \ldots, n$, we have $-1 \leq x_i - x'_i \leq 1$.  For each $i$, there are three possibilities.  First, suppose that we have $x_i - x'_i = 1$.  Then $y \sim_{S(X, k)} y'$ entails that, in the $i$th coordinates, we have
$$1 \geq kx_i + r_i - (kx'_i + r'_i) = k(x_i - x'_i) + (r_i - r'_i) = k + r_i - r'_i.$$
Thus $r'_i \geq k -1 + r_i$ and the only possibility is that, in this coordinate, we have $r_i = 0$ and $r'_i = k-1$.  From \defref{def: rho^c}, then, we have $s_i = 0$ and $s'_i = k-2$ and
hence  $(k-1)(x_i - x'_i) + (s_i - s'_i) = k-1 -(k-2) = 1$, which satisfies \eqref{eq: rho^c difference}.    Second, suppose that we have $x_i - x'_i = -1$.  Then
$$-1 \leq kx_i + r_i - (kx'_i + r'_i) = k(x_i - x'_i) + (r_i - r'_i) = -k + r_i - r'_i,$$
thus $r_i \geq k -1 + r'_i$, and  we have $r'_i = 0$ and $r_i = k-1$.  From \defref{def: rho^c}, then, we have $s'_i = 0$ and $s_i = k-2$ and
in this case   $(k-1)(x_i - x'_i) + (s_i - s'_i) = -(k-1) +(k-2) = -1$, which also satisfies \eqref{eq: rho^c difference}.   Finally, suppose that we have  $x_i - x'_i = 0$.  Here,
$y \sim_{S(X, k)} y'$ entails that we have
$$-1 \leq kx_i + r_i - (kx'_i + r'_i) = k(x_i - x'_i) + (r_i - r'_i) = 0 + r_i - r'_i \leq 1,$$
so that $r_i$ and $r'_i$ differ by at most $1$.  From \defref{def: rho^c}, if $\{r_i, r'_i\} \subseteq  [0, \lfloor k/2 \rfloor -1]$ or if $\{r_i, r'_i\} \subseteq  [\lfloor k/2 \rfloor, k -1]$, then we have
$(s_i - s'_i) = (r_i - r'_i)$ and so $-1 \leq s_i - s'_i \leq 1$.  The only other possibility is that we have  $\{r_i, r'_i\} =  \{ \lfloor k/2 \rfloor -1, \lfloor k/2 \rfloor \}$ in which case $s_i =s'_i$ and so $s_i - s'_i = 0$.   Wherever $r_i$ and $r'_i$ fall in $[0, k-1]$, then, we have $(k-1)(x_i - x'_i) + (s_i - s'_i) = (s_i - s'_i)$ which satisfies $|s_i - s'_i| \leq 1$ and \eqref{eq: rho^c difference} is again satisfied.  The result follows.
\end{proof}

\begin{corollary}\label{cor: factor rho}
Let $X \subseteq \Z^n$ be any digital image.  For any $k \geq 3$, we may factor the projection $\rho_k\colon S(X, k) \to X$ as
$$\rho_k = \rho_{k-1}\circ \rho^c_k\colon S(X, k) \to S(X, k-1) \to X,$$
with $\rho_{k-1}\colon S(X, k-1) \to X$ the standard projection and $\rho^c_k\colon S(X, k) \to S(X, k-1)$ the partial projection map from \defref{def: rho^c}.
\end{corollary}

\begin{proof}
It is sufficient to check that the composition agrees with  $\rho_k$ on $S(x, k)  \subseteq S(X, k)$, for each $x \in X$.  But  when restricted to  $S(x, k)$, both $\rho_k$ and
$\rho_{k-1}\circ \rho^c_k$ are constant maps.
\end{proof}

\section{One-Dimensional Domains: Paths and Loops in $Y$}\label{sec: 1D}

For $Y \subseteq \Z^n$ a digital image and any  $N \geq 1$, a \emph{path of length $N$ in $Y$} is a continuous map $\alpha\colon I_N \to Y$.  Unlike in the ordinary (topological) homotopy setting, where any path may be taken with the fixed domain $[0, 1]$, in the digital setting we must allow paths to have different domains.    Recall from \exref{ex: subdn} that we obtain a longer interval when we subdivide an interval: $S(I_N, k) = I_{Nk+k-1} \subseteq \Z$.

In the following result, notice that the map of subdivisions that covers the given path is itself a path (of length $N(2k+1)+2k$) in the subdivided digital image $S(Y, 2k+1)$.

\begin{theorem}\label{thm: path odd subdivision map}
Suppose we are given $\alpha\colon I_N \to Y$, a path of length $N$ in any digital image $Y \subseteq \Z^n$.  For any odd $2k+1 \geq 3$,  there is  a canonical choice of map of subdivisions
$$\widehat{\alpha}\colon S(I_N, 2k+1) = I_{N(2k+1)+2k}  \to S(Y, 2k+1)$$
that covers the given path, in the sense that the following diagram commutes:
$$\xymatrix{ S(I_N, 2k+1) \ar[d]_{\rho_{2k+1}} \ar[r]^-{\widehat{\alpha}} & S(Y, 2k+1) \ar[d]^{\rho_{2k+1}}\\
I_N \ar[r]_-{\alpha} & Y}$$
\end{theorem}

Our proof consists of an algorithmic construction of the covering map (or path) $\widehat{\alpha}$.  We first establish some notation and vocabulary used in the proof.
For the above diagram to commute, it is necessary and sufficient that the map $\widehat{\alpha}$ be a ``fibrewise" map, in the sense that it satisfies
\begin{equation}\label{eq: fibrewise}
\widehat{\alpha}\big( S(i, 2k+1)\big) \subseteq S(\alpha(i), 2k+1) \subseteq S(Y, 2k+1)
\end{equation}
for each $i\in I_N$ and $S(i, 2k+1) \subseteq S(I_N, 2k+1)$.
If $i \in I_N \subseteq \Z$ is a typical point in the interval,  write $\overline{i} = (2k+1)i+k \in S(i, 2k+1)$.  Thus, $\overline{i}$ is the point in the centre of the length $2k$ subinterval $S( i, 2k+1) \subseteq S(I_N, 2k+1)$ and, in particular, we have  $\rho_{2k+1}(\overline{i}) = i$, with $\rho_{2k+1} \colon S(I_N, 2k+1) \to I_N$ the standard projection. Then the $2k+1$ points of each $S( i, 2k+1)$ may be described as
\begin{equation}\label{eq: S( i, 2k+1)}
S( i, 2k+1) = \{ \overline{i} + r \mid -k \leq r \leq k \}.
\end{equation}
To describe points of $S(Y, 2k+1)$, we use notation similar to that used above in the discussion of covering an inclusion.
For each point  $y \in Y$, write $y = (y_1, \ldots, y_n)$ if $Y \subseteq \Z^n$.
Then, write $\overline{y} = \big((2k+1)y_1+k, \ldots, (2k+1)y_n+k\big) \in S(y, 2k+1)$, so that  $\overline{y}$ is the point in the centre of $S(y, 2k+1)$, which is a cubical
$(2k+1) \times (2k+1) \times \cdots \times (2k+1)$ lattice in $\Z^n$.  Namely, $S(y, 2k+1)$ is the translate of  $(I_{2k})^n\subseteq \Z^n$ by $(2k+1)y$.  Here, the scalar multiple $(2k+1)y$ means coordinate-wise (vector) scalar multiplication, and we will use coordinate-wise (vector) scalar multiplication and addition in $\Z^n$ freely in our notation.
Note, in particular, that we have  $\rho_{2k+1}(\overline{y}) = y$, with $\rho_{2k+1} \colon S(Y, 2k+1) \to Y$ the standard projection. Then the $(2k+1)^n$ points of each $S( y, 2k+1)$ may be described as
\begin{equation}\label{eq: S( y, 2k+1)}
S( y, 2k+1) = \{ \overline{y} + (r_1, \ldots, r_n) \mid -k \leq r_1, \ldots, r_n \leq k \}.
\end{equation}

\begin{proof}[Proof of \thmref{thm: path odd subdivision map}]
We define our covering map of subdivisions
$$\widehat{\alpha}\colon S(I_N, 2k+1)  \to S(Y, 2k+1)$$
in such a way so that we have
\begin{equation}\label{eq: centre to centre}
\widehat{\alpha}( \overline{i} ) = \overline{ \alpha(i) }
\end{equation}
for each $i \in I_N$.  That is, we will map the centre of the subinterval $S(i, 2k+1)$ to the centre of the cubical lattice $S( \alpha(i), 2k+1)$, for each $i$.  Now the key point to realize here is that, for any pair of adjacent points $y \sim_{\Z^n} y'$, the centres of $S(y, 2k+1)$ and $S(y', 2k+1)$ are joined by a (straight) segment of length $2k+1$, consisting of $2k+2$ points---including the two centres themselves as endpoints of the segment.  Of these $2k+2$ points, $k+1$ of them, including $\overline{y}$, are contained in $S(y, 2k+1)$ and  $k+1$ of them, including $\overline{y'}$, are contained in $S(y', 2k+1)$.  To define $\widehat{\alpha}$, then, we simply ``join the dots" between the centres of the cubical lattices, using the points of $S(I_N, 2k+1)$  between the centres of the subintervals to map point-for-point to the segments joining the centres of the lattices in $S(Y, 2k+1)$.

To define this map in symbols, which we do in formulas \eqref{eq: ends} and \eqref{eq: middle} (see also \eqref{eq: alt middle}) below, we write $\alpha(i+1) - \alpha(i)$ for the ``displacement vector" in $\Z^n$ from $\alpha(i)$ to $\alpha(i+1)$, for each $i = 0, \ldots, N-1$.  Since $\alpha(i) \sim_{\Z^n} \alpha(i+1)$, each coordinate of $\alpha(i+1) - \alpha(i)$ is $0$, $1$, or $-1$.  For each $j \in S(I_N, 2k+1)$ with  $k \leq j \leq N(2k+1) + k-1$, we may write
\begin{equation}\label{eq: quotient and remainder}
j -k = (2k+1) q(j)  + r(j),
\end{equation}
for unique $q(j) \in I_N$ and $r(j)$ with $0 \leq r(j) \leq 2k$.  Indeed, if $j$ falls in a subinterval of form
\begin{equation}\label{eq: centre to edge}
q(2k+1) +k \leq j \leq q(2k+1) + k + k, \text{ i.e., }  \overline{q} \leq j \leq \overline{q} + k,
\end{equation}
for some $q$ with $0 \leq q \leq N-1$, then we have $q(j) = q = \rho_{2k+1}(j) =  \lfloor j/(2k+1) \rfloor$ and $r(j)$ is in the range $0 \leq r(j) \leq k$.  On the other hand,   if $j$ falls in a subinterval of form
\begin{equation}\label{eq: edge to centre}
q(2k+1) +k + k+1 \leq j \leq q(2k+1) + k + 2k, \text{ i.e., }  \overline{(q+1)} - k \leq j \leq \overline{(q+1)} -1,
\end{equation}
for some $q$ with $0 \leq q \leq N-1$, then we have $q(j) = q$ but here $\rho_{2k+1}(j) =  \lfloor j/(2k+1) \rfloor = q+1$ and $r(j)$ is in the range $k+1 \leq r(j) \leq 2k$.

Also,  write $\mathbf{k} = (k, \ldots, k) \in \Z^n$ for  the vector each of whose coordinates is $k$.  Then, for each $i = 0, \ldots, N-1$, the two centres $\overline{\alpha(i)}$ and
$\overline{\alpha(i+1)}$ in $S(Y, 2k+1)$ that correspond to the adjacent points $\alpha(i)$ and $\alpha(i+1)$ in $Y$ have coordinates that satisfy
\begin{equation}\label{eq: displacement}
\begin{aligned}
\overline{\alpha(i+1)} - \overline{\alpha(i)} &= (2k+1) \alpha(i+1) + \mathbf{k} - \big( (2k+1) \alpha(i+1) + \mathbf{k}  \big) \\
&= (2k+1) \big[  \alpha(i+1) -  \alpha(i)  \big].
\end{aligned}
\end{equation}
Thus, we may pass from $\overline{\alpha(i)}$ to $\overline{\alpha(i+1)}$ by successively adding the displacement vector $ \alpha(i+1) -  \alpha(i)$ to $\overline{\alpha(i)}$ a total of $(2k+1)$ times. This is  the segment of points in $S(Y, 2k+1)$ joining the neighbouring centres alluded to above.

Our formula for $\widehat{\alpha}$, then is given as follows:
For $0 \leq j \leq N(2k+1) + 2k$, with the above notation, we define $\widehat{\alpha}$ on the parts of $S(I_N, 2k+1)$ before the first centre $\overline{0}$ and beyond the last centre $\overline{N}$ as
\begin{equation}\label{eq: ends}
\widehat{\alpha}(j) = \begin{cases}  \ \  \overline{\alpha(0)} & 0 \leq  j < k \\
 \\
\overline{\alpha(N)} & N(2k+1) + k \leq  j \leq N(2k+1) + 2k, \\
 \end{cases}
 \end{equation}
and on the part of $S(I_N, 2k+1)$ that falls between (any) centres as

\begin{equation}\label{eq: middle}
\widehat{\alpha}(j) =  \overline{\alpha\big(q(j)\big)} + r(j)\big[   \alpha\big(q(j)+1\big) - \alpha\big(q(j)\big)   \big]   \quad \text{for} \quad k \leq j < N(2k+1) +k.
 \end{equation}

\bigskip

We may also write \eqref{eq: middle} as follows, in a way that perhaps emphasizes the interpolation between centres.  First, write the domain of definition of \eqref{eq: middle} as  the disjoint union
$$[k, N(2k+1) + k-1] = \coprod_{i =0}^{N-1}\ [\overline{i}, \overline{i} +2k],$$
where we have $\overline{i} = i(2k+1) + k$, whence $\overline{i+1} =  \overline{i} + 2k$.  Also, for each $i$, if we write $j \in [\overline{i}, \overline{i} +2k]$ as $j = \overline{i} + t$ for some $t$ with $0 \leq t \leq  2k$, then $j + k = i(2k+1) + t$ and so $q(j) = i$ and $r(j) = t$ from \eqref{eq: quotient and remainder}.  Then, for each $i \in I_N$, \eqref{eq: middle} may also be written:

\begin{equation}\label{eq: alt middle}
\widehat{\alpha}(\overline{i} + t) =  \overline{\alpha(i)} + t\big[   \alpha(i+1) - \alpha(i)   \big]   \quad \text{for} \quad 0 \leq t \leq 2k.
 \end{equation}

First observe that this definition does indeed satisfy the ``centre-to-centre" property \eqref{eq: centre to centre}.   For if $i = 0, N$, Formula \eqref{eq: ends} gives  $\widehat{\alpha}( \overline{i} ) = \overline{ \alpha(i) }$.  If $1 \leq i \leq N-1$, then \eqref{eq: quotient and remainder} (or \eqref{eq: centre to edge}) gives $q(\overline{i}) = i$ and $r(\overline{i}) = 0$, whence Formula \eqref{eq: middle} (or \eqref{eq: alt middle}) gives $\widehat{\alpha}( \overline{i} ) = \overline{ \alpha(i) }$.

Next we confirm that, with this definition, the desired diagram commutes.  For this, we confirm that $\widehat{\alpha}$ has the fibrewise property of \eqref{eq: fibrewise}.   Divide $S(I_N, 2k+1)$ into a (disjoint) union of subintervals of the form
$$S(I_N, 2k+1) = \coprod_{i=0}^N\ [\overline{i} - k, \overline{i}  -1] \sqcup [\overline{i}, \overline{i} +k],$$
with the first type of subinterval consisting of the $k$ points to the left of a centre $\overline{i}$ and the second type consisting of the $k+1$ points to the right (including the centre $\overline{i}$ itself).  Note that we have
\begin{equation}\label{eq: S(i, 2k+1) subdivided}
S(i, 2k+1) = [\overline{i} - k, \overline{i}  -1] \sqcup [\overline{i}, \overline{i} +k]
\end{equation}
for each $i = 0, \ldots, N$ (see \eqref{eq: S( i, 2k+1)} above).

For $j \in [\overline{i} - k, \overline{i}  -1]$ with $1 \leq i \leq N$, formula \eqref{eq: edge to centre} and  the expressions that follow it give $q(j) = i-1$ and $r(j)$
in the range $k+1 \leq r(j) \leq 2k$.  From
formula \eqref{eq: middle}   we have
$$
\begin{aligned}
\widehat{\alpha}(j) &=  \overline{\alpha\big(i-1\big)} + r(j)\big[   \alpha\big(i\big) - \alpha\big(i-1\big)   \big] \\
&=  \overline{\alpha\big(i\big)} + \big(r(j) -(2k+1)\big)\big[   \alpha\big(i\big) - \alpha\big(i-1\big)   \big]
\end{aligned}
$$
where the re-write in the second line follows from  \eqref{eq: displacement}.  Since we have $-k \leq r(j) -(2k+1) \leq -1$ and the displacement vector $\alpha(i) - \alpha(i-1)$
has coordinates from $\{ 0, \pm 1\}$, it follows from \eqref{eq: S( y, 2k+1)} that we have
\begin{equation}\label{eq: left side fibrewise}
\widehat{\alpha}([\overline{i} - k, \overline{i}  -1])  \subseteq S( \alpha(i), 2k+1)
\end{equation}
for $i = 1, \dots, N$.

Similarly, for $j \in [\overline{i}, \overline{i} +k]$ with $0 \leq i \leq N_1$,   we have $q(j) = i$ and $0 \leq r(j) \leq k$ (cf.~\eqref{eq: centre to edge} above).  Then
$$\widehat{\alpha}(j) =  \overline{\alpha(i)} + r(j) (\alpha(i+1) - \alpha(i)) \in S( \alpha(i), 2k+1),$$
from \eqref{eq: S( y, 2k+1)}, because  the displacement vector $\alpha(i+1) - \alpha(i)$ has coordinates from $\{ 0, \pm 1\}$.   Hence, we also have
\begin{equation}\label{eq: right side fibrewise}
\widehat{\alpha}([\overline{i}, \overline{i} +k])  \subseteq S( \alpha(i), 2k+1)
\end{equation}
for $i = 0, \dots, N-1$.

Finally, Formula \eqref{eq: ends} gives directly that $\widehat{\alpha}([\overline{0} - k, \overline{0}  -1])  = \overline{ \alpha(0)} \in  S( \alpha(0), 2k+1)$ and $\widehat{\alpha}([\overline{N}, \overline{N} +k])  = \overline{ \alpha(N)} \in  S( \alpha(N), 2k+1)$.  These items combined with \eqref{eq: left side fibrewise},  \eqref{eq: right side fibrewise}, and \eqref{eq: S(i, 2k+1) subdivided}
confirm that  \eqref{eq: fibrewise} is satisfied for each $i \in I_N$.

For continuity, since $\widehat{\alpha}$ is a path in $S(Y, 2k+1)$, we simply need to check that $\widehat{\alpha}(j) \sim_{S(Y, 2k+1)} \widehat{\alpha}(j+1)$ for each $j = 0, \ldots, N(2k+1) + 2k-1$.  To this end, write
$S(I_N, 2k+1)$ as
a (disjoint) union of subintervals of the form
$$S(I_N, 2k+1) = [0, k-1] \sqcup \coprod_{i=0}^{N-1}\ [\overline{i}, \overline{i} + 2k] \sqcup [\overline{N}, N(2k+1) + 2k].$$
On each of these subintervals separately, $\widehat{\alpha}$ is easily seen to be  continuous.  In fact,  $\widehat{\alpha}$ is constant on the first and last.
Using \eqref{eq: quotient and remainder}--\eqref{eq: edge to centre}, we may write each of the remaining intervals as
\begin{equation}\label{eq: subint q and r}
\begin{aligned}
{[ \overline{i}, \overline{i} + 2k ] } &= [\overline{i}, \overline{i} + k] \sqcup  [\overline{(i+1)} - k, \overline{(i+1)} - 1]\\
&= \{ j \in S(I_N, 2k+1) \mid q(i) = i, r(j) = 0, \ldots, 2k \}.
\end{aligned}
\end{equation}
On $[ \overline{i}, \overline{i} + 2k ]$, then,  Formula \eqref{eq: middle} gives us
$$\widehat{\alpha}(j) = \overline{\alpha(i)} + r(j)[   \alpha(i+1) - \alpha(i) ]$$
with $r(j) = 0, \ldots, 2k$ as we take $j$ successively from $\overline{i}$ to $\overline{i} + 2k$.
Now each displacement vector $\alpha(i+1) - \alpha(i)$ has coordinates taken from $\{ 0, \pm 1\}$, and so when we add this term to a point in $\Z^n$, as we are doing here in passing from $\widehat{\alpha}(j)$ to $\widehat{\alpha}(j+1)$, we adjust each coordinate by at most $1$ to yield an adjacent point in $\Z^n$.

The remaining issue, then, is whether these continuous segments match-up in a continuous way.  For the first pair, namely $[0, k-1]$ and $[\overline{0}, \overline{0}+2k]$, we have $\widehat{\alpha}(k-1) = \widehat{\alpha}(\overline{0}) = \overline{\alpha(0)}$, and so $\widehat{\alpha}$ certainly gives a continuous map on the union of these two subintervals.   For the remaining pairs of adjacent subintervals, we must check that $\widehat{\alpha}(\overline{i} + 2k)$ and $\widehat{\alpha}(\overline{(i+1)})$ are adjacent, for $i = 0, \ldots, N-1$.   Using \eqref{eq: subint q and r} and the displayed formula below it for reference, we have
$$
\begin{aligned}
\widehat{\alpha}(\overline{(i+1)}) - \widehat{\alpha}(\overline{i} + 2k) &= \overline{\alpha(i+1)} - \big(  \overline{\alpha(i)} + 2k[   \alpha(i+1) - \alpha(i) ] \big) \\
&= (2k+1) \big[  \alpha(i+1) -  \alpha(i)  \big] - 2k [   \alpha(i+1) - \alpha(i)] \\
&=  \alpha(i+1)  - \alpha(i),
\end{aligned}
$$
where we have used \eqref{eq: displacement} to arrive at the middle line.  Once again we use the fact that each $\alpha(i+1) - \alpha(i)$ has coordinates taken from $\{ 0, \pm 1\}$ to conclude that $\widehat{\alpha}(\overline{i} + 2k)$ and $\widehat{\alpha}(\overline{(i+1)})$ are adjacent, and so $\widehat{\alpha}$ does indeed assemble into a continuous function.
\end{proof}

\begin{example}
Return to part (a) of \exref{ex: no subdivision map} and consider the map $f\colon I_1 \to I_1$ given by $f(0) = 1$ and $f(1) = 0$.  Applying \thmref{thm: path odd subdivision map}, we obtain a map
$$\widehat{\alpha} \colon S(I_1, 3) = I_5 \to S(I_1, 3) = I_5$$
that covers $f$.  It is given by $\widehat{\alpha}(i) = 5-i$ for $1 \leq i \leq 4$, and $\widehat{\alpha}(0) = 4$, $\widehat{\alpha}(5) = 1$.
\end{example}

\begin{example}
Let $\alpha\colon I_N \to Y$ be a constant path in $Y \subseteq \Z^n$.  Suppose that we have $\alpha(i) = y_0 \in Y$ for $0 \leq i \leq N$.  For any odd $2k+1$, the map $\widehat{\alpha}\colon S(I_N, 2k+1) \to S(Y, 2k+1)$ given by \thmref{thm: path odd subdivision map} that covers $\alpha$ is simply the constant path $\widehat{\alpha}(j) = \overline{y_0} \in S(Y, 2k+1)$ for $0 \leq j \leq (2k+1)N + 2k$.  For instance, we would cover the constant map in \exref{ex: no subdn interval} with the constant map $\widehat{\alpha}\colon S(I_1, 3) = I_5 \to S(I_0, 3) = I_2$ with $\widehat{\alpha}(i) = 1 \in I_2$ for each $i \in I_5$.
\end{example}

We will refer to the cover $\widehat{\alpha}$ of a path $\alpha$ constructed in \thmref{thm: path odd subdivision map} as the \emph{standard cover} of the path. Ideally, we would like to construct a functorial cover of maps of digital images regardless of the dimension of the domain, but we are not able to do so at present.  We observe here, though, that the standard cover of a path does have some functorial-like properties, such as the following:

\begin{lemma}\label{lem: technical stuff on covers}
Let $Y \subseteq \Z^n$ be any digital image.  For any path $\alpha\colon I_N \to Y$, let $\widehat{\alpha}$ denote the standard cover with respect to $(2k+1)$-fold subdivisions, so that
$\widehat{\alpha}$ makes the following diagram commute:
$$\xymatrix{ S(I_N, 2k+1) \ar[d]_{\rho_{2k+1}} \ar[r]^-{\widehat{\alpha}} & S(Y, 2k+1) \ar[d]^{\rho_{2k+1}}\\
I_N \ar[r]_-{\alpha} & Y}$$
\begin{itemize}
\item[(a)] If $C_N \colon I_N \to Y$ denotes the constant path at a point $y_0 \in Y$, then we have $\widehat{C_N} = C_{N'}\colon I_{N'} \to S(Y, 2k+1)$, the constant path at $\overline{y_0}$, where $N' = (2k+1)N + 2k$.
\item[(b)] If $Y = I_N$ and $\alpha\colon I_N \to I_N$ is the identity, then we have
$$\widehat{ \mathrm{id}_{I_N}} = \mathrm{id}_{S(I_N, 2k+1)} \colon S(I_N, 2k+1) \to S(I_N, 2k+1).$$
\end{itemize}
\end{lemma}

\begin{proof}
Both parts follow from a careful reading of the definition of $\widehat{\alpha}$.
\end{proof}

\begin{remark}\label{rem: even subdivisions 1D}
The conclusion of \thmref{thm: path odd subdivision map}  holds also for even subdivisions.  However the proof of this, whilst following essentially the same strategy as that of \thmref{thm: path odd subdivision map}, involves an adaptation to the fact that we have no ``middle points" in an even subdivision.  To avoid  giving another lengthy argument, much of which would be repetitive of the one just given, we settle instead for the weaker result below, which  is sufficient for our purposes here.

Still, we briefly indicate the way in which the proof of \thmref{thm: path odd subdivision map} may be adapted.  Recall that by an \emph{$n$-clique} in a digital image, we mean a set of $n$ points, each pair of which is adjacent.  For even subdivisions of $Y \subseteq \Z^n$, each cubical lattice $S(y, 2k)$ has a \emph{central $2^n$-clique} in place of the centre $\overline{y} \in  S(y, 2k+1)$.  For $i \in I_N$ an interval, $S(i, 2k)$ has a central $2$-clique, or middle pair.    To construct a covering map $F\colon S(I_N, 2k) \to S(Y, 2k)$,  we begin by mapping central $2$-cliques to central $2^n$-cliques (a choice is involved, which is determined by the ``displacement vectors" used in the proof of \thmref{thm: path odd subdivision map}), and then stringing these together using the remaining points of $S(I_N, 2k)$.    If we imagine our central $2^n$-cliques as ``lights" at the centre of each cubical lattice, then the covering paths here are akin to a string of (higher-dimensional) fairy lights,  with each light joined by a straight segment of wire.
\end{remark}

In the following, the conclusion for the case in which $k$ is odd is actually weaker than that of \thmref{thm: path odd subdivision map}.  We include it here so as to have a statement of the fact that a covering map exists independently of the parity of $k$.

\begin{corollary}\label{cor: path odd or even subdivision map}
Suppose we are given $\alpha\colon I_N \to Y$, a path of length $N$ in any digital image $Y \subseteq \Z^n$.  For any  $k \geq 2$,  there is  a map of subdivisions
$$F\colon S(I_N, k+1) = I_{N(k+1)+k}  \to S(Y, k)$$
that covers the given path, in the sense that the following diagram commutes:
$$\xymatrix{ S(I_N, k+1) \ar[d]_{\rho_{k+1}} \ar[r]^-{F} & S(Y, k) \ar[d]^{\rho_{k}}\\
I_N \ar[r]_-{\alpha} & Y}$$
\end{corollary}

\begin{proof}
Suppose that $k$ is even.  Pre-compose $\rho_{k}\colon S(Y, k) \to Y$ with the ``partial projection" $\rho^c_{k+1}\colon S(Y, k+1)  \to S(Y, k)$ of \defref{def: rho^c}. Then, as in  \corref{cor: factor rho}, we have $\rho_{k+1} = \rho_{k} \circ \rho^c_{k+1}\colon S(Y, k+1) \to Y$ and  \thmref{thm: path odd subdivision map} provides a filler for the diagram
$$\xymatrix{ S(I_N, k+1) \ar[d]_{\rho_{k+1}} \ar@{.>}[r]^-{\widehat{\alpha}} & S(Y, k+1) \ar[d]^{\rho_{k} \circ \rho^c_{k+1} }\\
I_N \ar[r]_-{\alpha} & Y.}$$
But then $F = \rho^c_{k+1}\circ \widehat{\alpha}\colon S(I_N, 2k+1) \to S(Y, k)$ provides the desired covering of $\alpha$.

Similarly, if  $k$ is odd, then use $F =\widehat{\alpha}\circ  \rho^c_{k+1}\colon S(I_N, k+1) \to S(Y, k)$.
\end{proof}

We end this section with a companion result about subdivision of \emph{loops} in a digital image.

\begin{definition}\label{def: loop}
A \emph{loop of length $N$} in a digital image $Y$ is a path $\gamma\colon I_N \to Y$ that satisfies $\gamma(0) = \gamma(N) \in Y$.
\end{definition}

\begin{corollary}\label{cor: loop subdivision map}
Suppose we are given $\gamma\colon I_N \to Y$, a loop of length $N$ in any digital image $Y \subseteq \Z^n$.  Suppose that we have $\gamma(0) = \gamma(N) = y_0 \in Y$. For any $k \geq 2$,  there is  a map of subdivisions
$$\Gamma\colon S(I_N, k') = I_{Nk'+k'-1}  \to S(Y, k),$$
with $k' \in \{ k, k+1\}$, that covers the given loop, in the sense that the following diagram commutes:
$$\xymatrix{ S(I_N, k') \ar[d]_{\rho_{k'}} \ar[r]^-{\Gamma} & S(Y, k) \ar[d]^{\rho_{k}}\\
I_N \ar[r]_-{\gamma} & Y}$$
Furthermore,  $\Gamma\colon S(I_N, k') \to S(Y, k)$ is a loop,  of length  $Nk'+k'-1$ in  $S(Y, k)$, and we may take $\Gamma$ to be a loop based at any point of
$S( y_0, k)$.
\end{corollary}

\begin{proof}
A review of the definitions of the covering paths in \thmref{thm: path odd subdivision map}  and \label{cor: path even subdivision map} reveals that  the standard cover of the loop $\gamma$ is a loop based at $\overline{y_0} \in  S( y_0, k)$ if $k$ is odd, or at $\rho^c_{k+1}(\overline{y_0})$, where $\overline{y_0} \in  S( y_0, k+1)$, if $k$ is even.  (Note that $k'$ is odd, whether $k$ is odd or even.)  In both results, the covering paths started and ended with a constant portion, of ``duration" equal to one-half the width of the appropriate cubical lattice.    For any $k$, rather than keep these ends constant, we treat them as ``loose ends," which then may be used so as to complete the loop at a different basepoint of $S( y_0, k)$ if desired.
\end{proof}

\section{Two-Dimensional Domains: Surfaces in $Y$}\label{sec: 2D}

We begin with a particular version of our main result.  We consider the case in which the domain is a rectangle $I_M \times I_N$.  In this case, we can give a rather clean and direct argument that generalizes the results of the previous section in a very satisfactory way.  Also, this case leads to a useful corollary about covers of homotopies (\corref{cor: path homotopy covering}), which we use in \cite{LOS19c}.  In the following proof, we rely heavily on the notation established for \thmref{thm: path odd subdivision map}.

\begin{theorem}\label{thm: 2-D subdivision map rectangle}
Suppose we are given a map $H \colon I_M \times I_N \to Y$ with $Y \subseteq \Z^n$ any  digital image.  For any $k \geq 1$,  there is  a canonical choice of map $\widehat{H}\colon S(I_M, 2k+1) \times S(I_N, 2k+1) \to S(Y, 2k+1)$ that makes the following diagram commute:
$$\xymatrix{ S(I_M, 2k+1) \times S(I_N, 2k+1) \ar[d]_{\rho_{2k+1} \times \rho_{2k+1} =\rho_{2k+1}} \ar[r]^-{\widehat{H}} & S(Y, 2k+1) \ar[d]^{\rho_{2k+1}}\\
I_M \times I_N \ar[r]_-{H} & Y}$$
Furthermore, if we define $\alpha(s) =  H(s, 0)$ and $\widetilde{\alpha}(t) = \widehat{H}(t, 0)$, then $\widetilde{\alpha} = \widehat{\alpha}\colon S(I_M, 2k+1) \to S(Y, 2k+1)$, the standard cover as in \thmref{thm: path odd subdivision map}  of the path $\alpha\colon I_M \to Y$.  Likewise along the other three edges of the rectangle $I_M \times I_N$.
\end{theorem}

\begin{proof}
For each $t$ with $0\leq t \leq N$, define $\alpha_t\colon I_M \to Y$, and for each $s$ with $0\leq s \leq M$, define $\beta_s\colon I_N \to Y$, as
$$\alpha_t(s) = H(s, t), \text{ for } 0 \leq s \leq M \quad \text{and} \quad \beta_s(t) = H(s, t), \text{ for } 0 \leq t \leq N.$$
So the $\alpha_t$ are the horizontal coordinate curves of $H$, and  the $\beta_s$ are the vertical.  Then as in \thmref{thm: path odd subdivision map}, each of these paths has a standard cover
$$\widehat{\alpha_t}\colon S(I_M, 2k+1) = I_{(2k+1)M + 2k} \to S(Y, 2k+1)$$
and
$$\widehat{\beta_s}\colon S(I_N, 2k+1) = I_{(2k+1)N + 2k} \to S(Y, 2k+1).$$
We will define $\widehat{H}$ in such a way as to have these be amongst the horizontal and vertical coordinate curves of $\widehat{H}$, respectively.

Recall from our generalities on subdivision in \secref{sec: subdivision} that we have an isomorphism of digital images $S(I_M \times I_N, 2k+1)  \cong S(I_M, 2k+1) \times S(I_N, 2k+1)$.  For individual points $(i, j) \in  I_M \times I_N$, we may specialize this identification to an isomorphism  $S\big((i, j), 2k+1\big)  \cong S(i, 2k+1) \times S(j, 2k+1)$.  We use these identifications repeatedly in what follows.

Recall also from \thmref{thm: path odd subdivision map} that, for $i \in I_M$, we write the centre of the subinterval $S(i, 2k+1) \subseteq S(I_M, 2k+1)$ as $\overline{i} = i(2k+1) + k$.  Then each $(2k+1)\times(2k+1)$ sub-lattice $S\big( (i, j), 2k+1\big) \subseteq S( I_M\times I_N, 2k+1)$ has the point
$$\overline{(i, j)} = (\overline{i}, \overline{j}) = (i(2k+1) + k, j(2k+1) + k)$$
at its centre.  We refer to these points as \emph{centres} of the sub-divided digital image $S( I_M\times I_N, 2k+1)$.  Furthermore, for a point $y \in Y$, we write $\overline{y}$ for the centre of $S(y, 2k+1) \subseteq S(Y, 2k+1)$.

We may begin by defining  $\widehat{H}$ on these centres as
\begin{equation}\label{eq: H cover on centres}
\widehat{H}\big(\overline{(i, j)}\big)  = \widehat{H}(\overline{i}, \overline{j})  := \overline{H(i, j)},
\end{equation}
for each $(i, j) \in I_M \times I_N$.  We will extend this definition of $\widehat{H}$ over the whole of $S( I_M\times I_N, 2k+1)$ in several steps.

\subsection{Step 1: Outside the centres}
For  $s < \overline{0} =  k$ or $t < \overline{0} =  k$, or $s > \overline{M} = (2k+1)M + k$ or $t > \overline{N} = (2k+1)N + k$ define
$$
\begin{aligned}
\widehat{H}(s, t) &= \widehat{\beta_0}(t)  \text{ for } 0 \leq s \leq k-1 \text{ and } 0 \leq t \leq (2k+1)N + 2k  \\
\widehat{H}(s, t) &= \widehat{\alpha_0}(s)  \text{ for } 0 \leq s \leq (2k+1)M + 2k  \text{ and } 0 \leq t \leq k-1   \\
\widehat{H}(s, t) &= \widehat{\beta_M}(t)  \text{ for } \overline{M}+1 \leq s \leq (2k+1)M + 2k \text{ and } 0 \leq t \leq (2k+1)N + 2k  \\
\widehat{H}(s, t) &= \widehat{\alpha_N}(s)  \text{ for } 0 \leq s \leq (2k+1)M + 2k \text{ and }  \overline{N}+1 \leq t \leq (2k+1)N + 2k.
\end{aligned}
$$
The situation is illustrated in \figref{fig:Th.5.1 Step 1}.  Dots represent the points on which $\widehat{H}$ has been defined at this point. Solid dots represent centres, on which we have defined $\widehat{H}$ as in \eqref{eq: H cover on centres}.  Open dots are those points on which we have defined $\widehat{H}$ at this step.   We have also included some gridlines (dotted) in the figure.  These gridlines do not pass through points (they are not gridlines of the integer lattice).  Rather, they pass between points, and serve to aggregate points into $(2k+1) \times (2k+1)$ squares in $S(I_M \times I_N, 2k+1)$, of the form
$$S\big( (i, j), 2k+1\big) = [(2k+1)i, (2k+1)i +2k] \times [(2k+1)j, (2k+1)j +2k],$$
for $(i, j) \in I_M \times I_N$.  Each of these squares contains one center, namely $\overline{(i, j)} \in S\big( (i, j), 2k+1\big)$.  All points in one of these squares are mapped to one point of $I_M \times I_N$ by the standard projection; we have $\rho_{2k+1}\big( S\big( (i, j), 2k+1\big) \big) = (i, j) \in I_M \times I_N$.
\begin{figure}[h!]
\centering
\includegraphics[trim=120 440 120 130,clip,width=\textwidth]{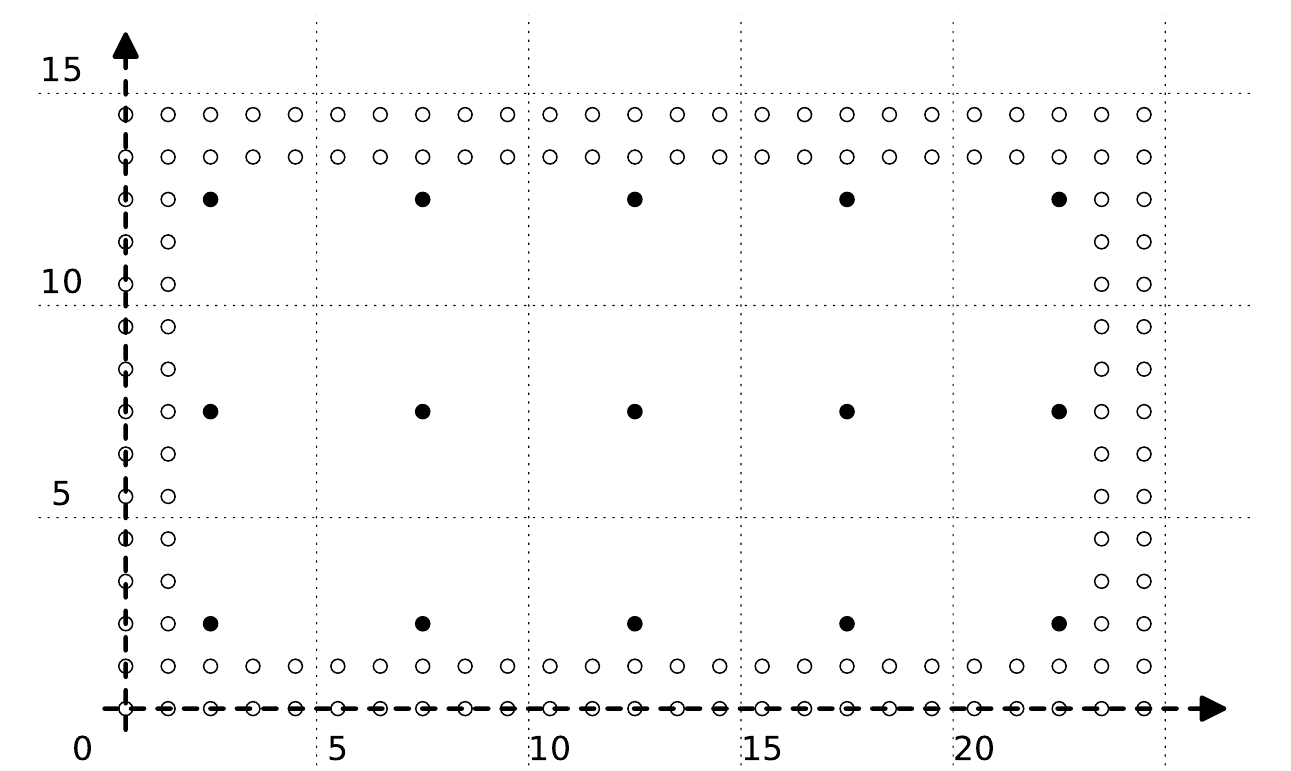}
\caption{\label{fig:Th.5.1 Step 1} $\widehat{H}$ after Step 1.  Illustrated with $M = 4$, $N=2$, and $2k+1 = 5$.}
\end{figure}

Notice that where definitions from this step overlap with each other, namely in each of the four corner regions, the definitions agree.  For example, if $0 \leq s, t \leq k-1$, we have  $\widehat{H}(s, t) = \widehat{\beta_0}(t)$ and $\widehat{H}(s, t) = \widehat{\alpha_0}(s)$.  Now, for $0 \leq t \leq k-1$, \thmref{thm: path odd subdivision map} gives $\widehat{\beta_0}(t) = \overline{ \beta_0(0)} = \overline{ H(0,0)}$, and similarly we have  $\widehat{\alpha_0}(s) = \overline{ H(0,0)}$  for $0 \leq s \leq k-1$.  The other four corner regions behave similarly.

We will check continuity after the next step.

\subsection{Step 2: Coordinate curves through the centres}  Next we extend the definition of $\widehat{H}$ to  the horizontals and verticals through each centre of $S(I_M \times I_N, 2k+1)$.  On these, we define for each $i \in I_M$ and $j \in I_N$,
$$
\begin{aligned}
\widehat{H}(\overline{i}, t) &= \widehat{\beta_i}(t)  \text{ for }  0 \leq t \leq (2k+1)N + 2k  \\
\widehat{H}(s, \overline{j}) &= \widehat{\alpha_j}(s)  \text{ for } 0 \leq s \leq (2k+1)M + 2k.
\end{aligned}
$$
The situation is illustrated in \figref{fig:Th.5.1 Step 2}.  Again,  dots represent the points on which $\widehat{H}$ has now been defined.  Solid dots represent centres; open dots represent points on which the definition of  $\widehat{H}$ has been extended in Steps 1 and 2.
\begin{figure}[h!]
\centering
\includegraphics[trim=120 440 120 130,clip,width=\textwidth]{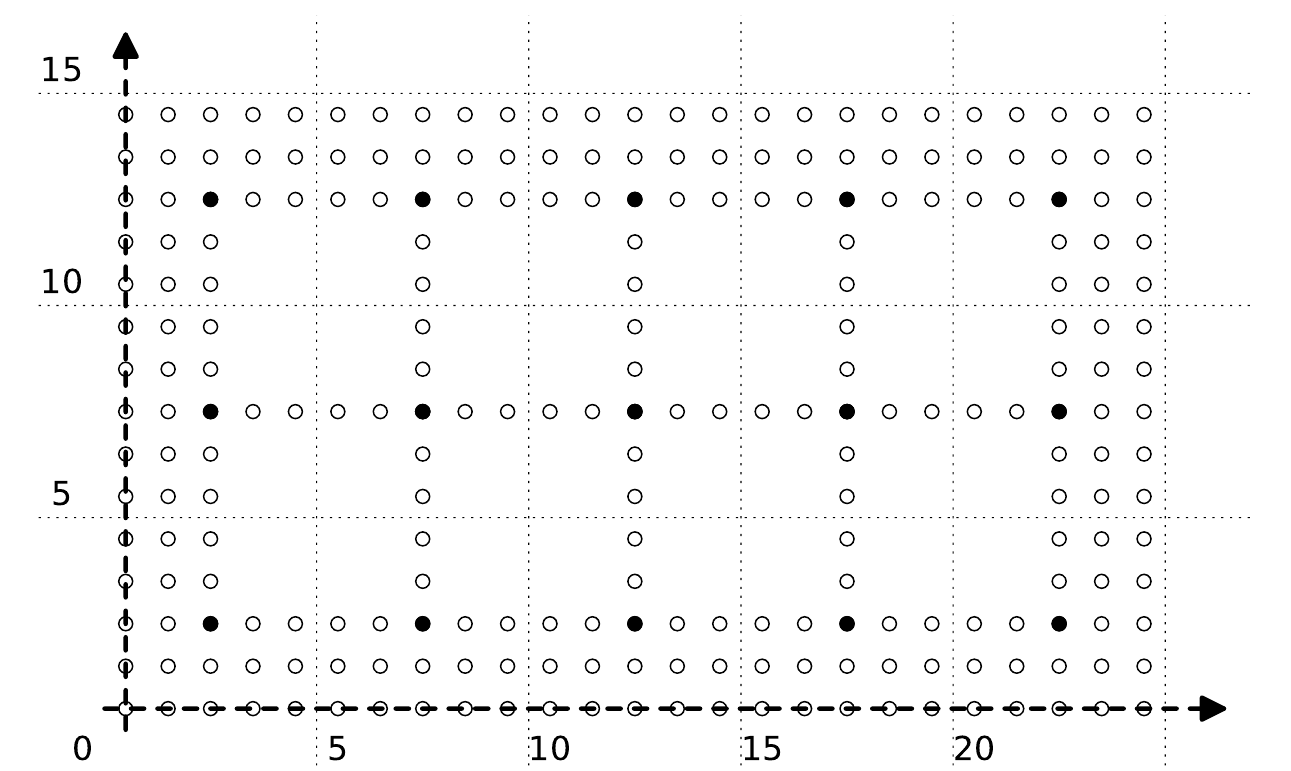}
\caption{\label{fig:Th.5.1 Step 2} $\widehat{H}$ after Step 2.  Illustrated with $M = 4$, $N=2$, and $2k+1 = 5$.}
\end{figure}
We check $\widehat{H}$ is well-defined. In any horizontal row or vertical column that includes centres, this Step 2 includes a definition of $\widehat{H}$ at those centres.  Notice  that  the way in which we defined the standard cover  of a path in \thmref{thm: path odd subdivision map} extended the ``centre-to-centre" definition of \eqref{eq: centre to centre}, so the value assigned to $\widehat{H}$ on any centre at this step is consistent with the value assigned by \eqref{eq: H cover on centres}.  The only other overlap in definition is  at the top or bottom of a vertical, or the left and right ends of a horizontal.  For example, if $0 \leq t \leq k-1$, we have $\widehat{H}(\overline{i}, t) = \widehat{\beta_i}(t)$ from this step, and
$\widehat{H}(\overline{i}, t) = \widehat{\alpha_0}(\overline{i})$ from Step 1.  Now  $\widehat{\beta_i}(t) = \overline{\beta_i(0)}$, since $t \leq k-1$, so we have $\widehat{\beta_i}(t) = H(i, 0)$.  But $\widehat{\alpha_0}(\overline{i}) = H(i, 0)$, and the definitions agree.  The other overlaps around the edges are seen to agree similarly;  $\widehat{H}$ is well-defined thus far.

Now we  check continuity, so far as we have defined $\widehat{H}$.  To this end, suppose we have adjacent points $(s, t)$ and $(s', t')$ in that part of $S( I_M\times I_N, 2k+1)$ on which we have defined $\widehat{H}$.  If both points are in one of the horizontal bands $0 \leq t, t' \leq \overline{0}=k$ or $(2k+1)N + k = \overline{N}  \leq t, t' \leq (2k+1)N + 2k$, or if both points are in one of the horizontal rows through centres $t = t' = \overline{j}$ for some $j$ with $j \in I_N$,    then adjacency of $\widehat{H}(s, t)$ and $\widehat{H}(s', t')$ in $S(Y, 2k+1)$ follows immediately from the continuity of the standard covers $\widehat{\alpha_j}$.  This is because, on these horizontal regions, we have defined $\widehat{H}(s, t) = \widehat{\alpha_j}(s)$, for a suitable $j$ depending on $t$.  Hence, for  $(s, t) \sim (s', t')$, we have $s \sim s'$ in $S(I_M, 2k+1)$, whence $\widehat{\alpha_j}(s) \sim \widehat{\alpha_j}(s')$  and therefore  $\widehat{H}(s, t) \sim \widehat{H}(s', t')$.
For both points in one the
vertical bands $0 \leq s, s' \leq \overline{0}=k$ or $(2k+1)M + k = \overline{M}  \leq s, s' \leq (2k+1)M + 2k$, or both points  in one of the vertical columns through centres $s = s' = \overline{i}$ for some  $i \in I_M$,  adjacency of $\widehat{H}(s, t)$ and $\widehat{H}(s', t')$ in $S(Y, 2k+1)$ follows immediately from the continuity of the standard covers $\widehat{\beta_i}$, in a similar way.

It remains to consider the cases in which  one point lies in a horizontal row or band, the other point lies in a vertical row or band, and they are  situated ``across a corner from each other" so that both do not lie in a horizontal or a vertical.  This entails that one point is on a horizontal and one on a vertical, each adjacent, but not equal, to a centre $(\overline{i}, \overline{j})$ (see \figref{fig:Th.5.1 Step 2}).  For example, consider a pair $(\overline{i}+1, \overline{j})$ and $(\overline{i}, \overline{j}+1)$.  Here, we have
$$
\begin{aligned}
\widehat{H}(\overline{i}+1, \overline{j}) &=   \widehat{\alpha_j}( \overline{i}+ 1) =  \overline{ \alpha_j(i)} + 1\cdot[\alpha_j(i+1) - \alpha_j(i)] \\
&= \overline{ H(i, j)}+ 1\cdot[H(i+1, j) - H(i, j)]
\end{aligned}
$$
and
$$
\begin{aligned}
\widehat{H}(\overline{i}, \overline{j}+1) &=   \widehat{\beta_i}( \overline{j}+ 1) =  \overline{ \beta_i(j)} + 1\cdot[\beta_i(j+1) - \beta_i(j)] \\
&= \overline{ H(i, j)}+ 1\cdot[H(i, j+1) - H(i, j)].
\end{aligned}
$$
The difference between these two, using vector arithmetic in $S(Y, 2k+1)$, is
$$\widehat{H}(\overline{i}+1, \overline{j}) - \widehat{H}(\overline{i}, \overline{j}+1) = H(i+1, j) - H(i, j+1) \in \Z^n.$$
Since $H$ is continuous, and $(i+1, j) \sim (i, j+1)$ in $I_M \times I_N$, each coordinate of this difference is in $\{0, \pm1\}$, and it follows that we have
$$\widehat{H}(\overline{i}+1, \overline{j}) \sim_{S(Y, 2k+1)} \widehat{H}(\overline{i}, \overline{j}+1).$$
Similarly, consider the pair of points  $( \overline{M}-1, k) \sim ( \overline{M}, k+1) \in S( I_M\times I_N, 2k+1)$    (towards the lower-right corner in \figref{fig:Th.5.1 Step 2}).  Now  $( \overline{M}-1, k) =  ( \overline{M-1}+2k, k)$, so we have (cf.~formula \eqref{eq: alt middle} from \thmref{thm: path odd subdivision map})
$$
\begin{aligned}
\widehat{H}( \overline{M}-1, k) &=   \widehat{\alpha_0}( \overline{M-1}+ 2k) =  \overline{ \alpha_0(M-1)} + 2k[\alpha_0(M) - \alpha_0(M-1)] \\
&= \overline{ H(M-1, 0)}+ 2k[H(M, 0) - H(M-1, 0)]
\end{aligned}
$$
and, since we have $k = \overline{0}$ (refer again to  \eqref{eq: alt middle})
$$
\begin{aligned}
\widehat{H}( \overline{M}, k+1) &=   \widehat{\beta_M}( k+1) =  \widehat{\beta_M}( \overline{0}+1) = \overline{ \beta_M(0)}+ [\beta_M(1) - \beta_M(0)] \\
&= \overline{ H(M, 0)} + [H(M, 1) - H(M, 0)].
\end{aligned}
$$
Using vector arithmetic in $S(Y, 2k+1)$,
we may write  $\overline{ H(M-1, 0)} = (2k+1)H(M-1, 0) + (k, k)$ and $\overline{ H(M, 0)} = (2k+1)H(M, 0) + (k, k)$.   The difference between  $\widehat{H}( \overline{M}-1, k)$ and $\widehat{H}( \overline{M}, k+1)$, then, is
$$\widehat{H}( \overline{M}, k+1) - \widehat{H}( \overline{M}-1, k)  =  H(M, 1)- H(M-1, 0).$$
Now $(M, 1) \sim_{I_M\times I_N} (M-1, 0)$, and hence $H(M, 1)\sim_Y H(M-1, 0)$ from the continuity of $H$.  It follows that each coordinate of $H(M, 1)- H(M-1, 0) \in \Z^n$, and hence each coordinate of  $\widehat{H}( \overline{M}, k+1) - \widehat{H}( \overline{M}-1, k)$, belongs to $\{ 0, \pm1\}$.  That is, we have
$$\widehat{H}( \overline{M}, k+1) \sim_{S(Y, 2k+1)} \widehat{H}( \overline{M}-1, k).$$
Other cases are checked similarly; we leave the details as an exercise.  It follows that $\widehat{H}$ is continuous, so far as we have defined it.

\subsection{Step 3: Extension over squares whose corners are centres} The last step requires some ideas beyond those of \thmref{thm: path odd subdivision map}.  But, first, note that we may extend $\widehat{H}$ over the interior of any square in $S(I_M\times I_N, 2k+1)$ whose corners are centres independently of any other such square.  This is because any two points of $S(I_M\times I_N, 2k+1)$ that are adjacent must be in one such square (including its edges) or, if not, then both must be in the region of  $S(I_M\times I_N, 2k+1)$ from Part 2, where we have already confirmed continuity.  So it is sufficient to show that we may extend $\widehat{H}$ over a typical such square
$$[\overline{i}, \overline{i+1}] \times [\overline{j}, \overline{j+1}]$$
with corners
$$\{ (\overline{i}, \overline{j}), (\overline{i+1}, \overline{j}), (\overline{i}, \overline{j+1}), (\overline{i+1}, \overline{j+1}) \}$$
for some $(i, j) \in I_M\times I_N$.  Such a square is illustrated in \figref{fig:Th.5.1 Step 3}. As in the two previous figures, dots indicate points on which we have already defined $\widehat{H}$.  Notice we have preserved portions of the gridlines discussed when we described the features of \figref{fig:Th.5.1 Step 1} above.  These gridlines now divide each square  $[\overline{i}, \overline{i+1}] \times [\overline{j}, \overline{j+1}]$ into four quadrants.  Each quadrant contains a centre of $S( I_M \times I_N, 2k+1)$ (at its corner) and all points in one quadrant are mapped to one point of the $4$-clique $[i, i+1]\times [j, j+1] \subseteq I_M \times I_N$ by the standard projection. It follows that, if we are to cover $H$, the image under $\widehat{H}$ of all points in one of these quadrants must lie in some $S(y, 2k+1) \subseteq S(Y, 2k+1)$ for a single point $y \in Y$.  For example, The lower-left quadrant of  $[\overline{i}, \overline{i+1}] \times [\overline{j}, \overline{j+1}]$ consists of the points $[\overline{i}, \overline{i} +k] \times [\overline{j}, \overline{j}+k]$ and we require the extended $\widehat{H}$ to satisfy
$$\widehat{H} \big( [\overline{i}, \overline{i} +k] \times [\overline{j}, \overline{j}+k]\big) \subseteq S\big( H(i, j), 2k+1\big) \subseteq S(Y, 2k+1).$$
In the statement of  \lemref{lem: unit square covering} below, this ``quadrant-wise" behaviour of an extension to a cover is addressed explicitly.  Furthermore, as we progress with the proof of \lemref{lem: unit square covering}, we will depend heavily on having the square divided into quadrants in this way.

In the previous steps, we have already defined $\widehat{H}$ on the edges and corners of this square.
\begin{figure}[h!]
\centering
\includegraphics[trim=280 490 120 130,clip,width=0.5\textwidth]{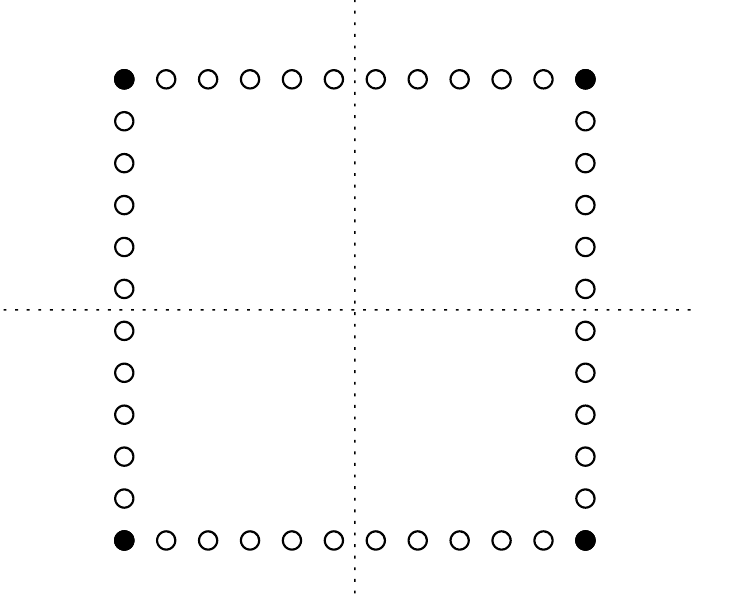}
\caption{\label{fig:Th.5.1 Step 3} A typical square over which we extend $\widehat{H}$ in Step 3.  Illustrated with $2k+1 = 11$.}
\end{figure}
We will apply \lemref{lem: unit square covering} below to extend over the interior of this square.  To do so, use the given $H$ to determine a unit $n$-cube as follows.  On each point of
$[i, i+1]\times [j, j+1]$ write $H$ coordinate-wise as
$$H(i, j) = \big( H_1(i, j), \ldots, H_n(i, j)\big),$$
and so-on for the other points. Then, for each coordinate $r = 1, \ldots, n$, set
$$a_r = \mathrm{min} \{ H_r(i, j), H_r(i+1, j) , H_r(i, j+1), H_r(i+1, j+1) \},$$
and let $A$ denote the unit $n$-cube
$$A = [a_1, a_1+1] \times \cdots \times [a_n, a_n+1] \subseteq\Z^n.$$
Then write
$$\overline{A} = [\overline{a_1}, \overline{a_1 + 1}] \times [\overline{a_2}, \overline{a_2 + 1}] \times \cdots \times [\overline{a_n}, \overline{a_n + 1}],$$
so that $\rho_{2k+1}\colon  \overline{A} \to A$ maps each orthant of $\overline{A}$ to the corresponding corner of $A$.  If, as in  \lemref{lem: unit square covering}, we write
$$\partial\big( [\overline{i}, \overline{i+1}] \times [\overline{j}, \overline{j+1}]\big)$$
for the boundary of the square $[\overline{i}, \overline{i+1}] \times [\overline{j}, \overline{j+1}]$, then from Steps 1 and 2 we have
$$\widehat{H}\colon  \partial\big( [\overline{i}, \overline{i+1}] \times [\overline{j}, \overline{j+1}]\big) \to \overline{A},$$
with each corner of $[\overline{i}, \overline{i+1}] \times [\overline{j}, \overline{j+1}]$  mapped by $\overline{H}$ to some corner of $\overline{A}$, and each edge of $[\overline{i}, \overline{i+1}] \times [\overline{j}, \overline{j+1}]$  mapped to the corresponding edges or diagonals of $\overline{A}$.  Notice this ``corner-to-corner" assertion follows from our choice of the coordinates for the distinguished ``minimal" corner $(a_1, \ldots, a_n)$ of $A$: Because $a_r$ is the minimum of $\{ H_r(i, j), H_r(i+1, j) , H_r(i, j+1), H_r(i+1, j+1) \}$, and because $H$ is continuous, it follows that we have
$$a_r \leq   H_r(i, j), H_r(i+1, j) , H_r(i, j+1), H_r(i+1, j+1) \leq a_r+1$$
for each point of the $4$-clique $[i, i+1] \times [j, j+1]$.
 Notice also that some of the corners of the $n$-cube $A$, respectively $\overline{A}$, may lie outside $Y$, respectively $S(Y, 2k+1)$.  The image of our square
$[i, i+1]\times[j, j+1]$ under $H$, however, does lie in $Y$ and it will follow that  the image of
$[\overline{i}, \overline{i+1}] \times [\overline{j}, \overline{j+1}]$
 under the extended $\widehat{H}$ likewise will be contained in $S(Y, 2k+1)$.

Now define translations in $\Z^2$ by
$$\overline{T_1} (\overline{i}, \overline{j}) = (\overline{0}, \overline{0}) \quad \text{and} \quad  T_1 (i, j) = (0, 0),$$
and translations in $\Z^n$ by
$$\overline{T_2} (\overline{a_1}, \ldots,  \overline{a_n}) = (\overline{0}, \ldots, \overline{0}) \quad \text{and} \quad  T_2 (a_1, \ldots, a_n) = (0, \ldots, 0).$$
Translation $\overline{T_1}$ preserves the boundary of the square; both pairs of translations respect standard projections, in that we have
$$\rho_{2k+1}\circ \overline{T_1} = T_1 \circ \rho_{2k+1}\colon [\overline{i}, \overline{i+1}] \times [\overline{j}, \overline{j+1}] \to [0, 1]^2$$
and
$$\rho_{2k+1}\circ \overline{T_2} = T_2 \circ \rho_{2k+1}\colon \overline{A} \to [0, 1]^n.$$
Apply \lemref{lem: unit square covering} to the map
$$F:=  \overline{T_2}\circ   \widehat{H} \circ (\overline{T_1})^{-1}\colon \partial\big( [\overline{0}, \overline{1}]^2 \big) \to [\overline{0}, \overline{1}]^n$$
with  the map $f \colon [0, 1]^2 \to [0, 1]^n$ defined by either  $f(p, q) = f\big(\rho_{2k+1}(\overline{p}, \overline{q})\big) = \rho_{2k+1}\big( F(\overline{p}, \overline{q})\big)$ or $f(p, q) = T_2\circ H \circ (T_1)^{-1}(p, q)$, since these agree on  $[0, 1]^2$.  The result is an extension of $F$ to $[\overline{0}, \overline{1}]^2$ which we may use to extend $\widehat{H}$ from the boundary of  $[\overline{i}, \overline{i+1}] \times [\overline{j}, \overline{j+1}]$ to the map
$$\widehat{H}:=    (\overline{T_2})^{-1}\circ  F \circ \overline{T_1} \colon [\overline{i}, \overline{i+1}] \times [\overline{j}, \overline{j+1}] \to \overline{A}.$$
This extension fits into the following diagram, in which all parts commute and in which we may reverse the directions of the translations and preserve commutativity:
$$\xymatrix{
 [\overline{i}, \overline{i+1}] \times [\overline{j}, \overline{j+1}] \ar[rrr]^-{\widehat{H}} \ar[rd]_-{\rho_{2k+1}} \ar[ddd]_{\overline{T_1}}& & &  \overline{A} \ar[ddd]^{\overline{T_2}} \ar[ld]^{\rho_{2k+1}}\\
 & [i, i+1]\times[j, j+1] \ar[r]^-{H} \ar[d]_{T_1} & A \ar[d]^{T_2}\\
 & [0,1]^2 \ar[r]_-{f} & [0,1]^n \\
 [\overline{0}, \overline{1}]^2 \ar[rrr]_-{F} \ar[ru]^{\rho_{2k+1}}& & & [\overline{0}, \overline{1}]^n  \ar[lu]^{\rho_{2k+1}}
}$$
The fact that the top trapezoid commutes means that, although $\overline{A}$ may contain points outside $S(Y, 2k+1)$, nonetheless the image of the extended $\widehat{H}$ must be contained in $S(Y, 2k+1)$, since the image of the original $H$ is contained in $Y$. As we remarked previously, it is sufficient to be able to extend $\widehat{H}$ over each square such as $ [\overline{i}, \overline{i+1}] \times [\overline{j}, \overline{j+1}] $ one at a time to complete the proof.
\end{proof}

It remains to prove the special case used at the heart of Step 3 of the above proof.  With reference to a $(2k+1)$-fold subdivision of either $\Z^2$ or $\Z^n$, Write the boundary of the square $[\overline{0}, \overline{1}] \times [\overline{0}, \overline{1}] = [k, 3k+1]\times [k, 3k+1]  \subseteq \Z^2$ as
$$\partial\big([\overline{0}, \overline{1}] \times [\overline{0}, \overline{1}]\big) = \{  \overline{0},  \overline{1}\} \times [\overline{0}, \overline{1}]
\cup [\overline{0}, \overline{1}]\times \{ \overline{0}, \overline{1}\},$$
and suppose that we have a map
$$F\colon \partial\big([\overline{0}, \overline{1}] \times [\overline{0}, \overline{1}]\big)  \to ([\overline{0}, \overline{1}] )^n
 \subseteq \Z^n$$
with the following two properties:

\subsection{Property (1)}\label{subs: Prop 1} $F$ preserves corners.  Namely,  we have
$$F\big(  \{ (\overline{0}, \overline{0}), (\overline{1}, \overline{0}), (\overline{0}, \overline{1}), (\overline{1}, \overline{1})\} \big) \subseteq \big\{ (\overline{i_1}, \ldots, \overline{i_1}) \mid \{ i_1, \ldots, i_n \} \subseteq \{0, 1\} \big\}.$$
Any map $F$ that possesses this property allows us to define a map $f\colon [0, 1]^2 \to [0, 1]^n$ as $f(i, j) = \rho_{2k+1}\circ F(\overline{i}, \overline{j})$ for $(i, j) \in [0, 1]\times [0,1]$, and then view $F$ as an extension over the boundary of $[\overline{0}, \overline{1}]^2$ of a cover of $f$.

\subsection{Property (2)}\label{subs: Prop 2} $F$ also interpolates edges to edges or diagonals.  That is, suppose $v, v' \in [0, 1]^2$ are either horizontal or vertical neighbours (not diagonal neighbours), so that  $\overline{v}, \overline{v'} \in [\overline{0}, \overline{1}]^2$ are two corners  at either end of a horizontal or vertical  edge of $[\overline{0}, \overline{1}]^2$.  \emph{Per} Property  (1), $F(\overline{v}), F(\overline{v'})$ are corners of $([\overline{0}, \overline{1}] )^n$  and $f(v), f(v')$ the corresponding corners of the unit square $[0, 1]^n \subseteq \Z^n$, where $f(v) = \rho_{2k+1}\big( F(\overline{v})  \big)$ and $f(v') = \rho_{2k+1}\big( F(\overline{v'}) \big)$ (notice that these may no longer be at either end of an edge, though). Parametrize the edge from $\overline{v}$ to $\overline{v'}$ as
$$ \{ \overline{v} + t [ v' - v] \} \text{ for } t = 0, \ldots, 2k+1.$$
Then  along each edge of $[\overline{0}, \overline{1}]^2$, Property (2) requires that we have
\begin{equation}\label{eq: F on edge}
F\big(  \overline{v} + t [v' - v]  \big) =
F(\overline{v}) + t [ f(v') - f(v)] \text{ for } t = 0, \ldots, 2k+1
\end{equation}
where, once again,  $f$ is defined as $f(v) = \rho_{2k+1}\circ F(\overline{v})$ for $v \in [0, 1]^2$.
Property (2), together with the fact that $F$ is mapping into a cube, entails that the given $F$ must be continuous.  It is easy to recognize the situation of Step 3 of the above proof here.  We have a commutative diagram
$$\xymatrix{ \partial\big( [\overline{0}, \overline{1}]\times [\overline{0},\overline{1}] \big) \ar[r]^-{F} \ar[d]_{\rho_{2k+1}}
 & ([\overline{0}, \overline{1}] )^n  \ar[d]^{\rho_{2k+1}} \\
[0, 1]\times [0,1] \ar[r]_-{f} &   ([0, 1])^n.}$$
Along the edges of
$[\overline{0}, \overline{1}]\times [\overline{0},\overline{1}]$, the map $F$ agrees with the standard covers of the unit-length paths in  $([0, 1])^n$ given by restricting $f$ to the unit-length edges of the unit square. (Note that $[\overline{0}, \overline{1}]\times [\overline{0},\overline{1}]$ is a sub-square of $S([0, 1]^2, 2k+1)$, however.)

\begin{lemma}\label{lem: unit square covering}
With the above notation, a map
$$F\colon \partial\big( [\overline{0}, \overline{1}]\times [\overline{0},\overline{1}] \big) \to [\overline{0}, \overline{1}]^n \subseteq \Z^n$$
that satisfies Properties (1) and (2), of (\ref{subs: Prop 1}) and (\ref{subs: Prop 2}), may be extended in a canonical way to a continuous map
$F\colon [\overline{0}, \overline{1}]^2 \to [\overline{0}, \overline{1}]^n$
that makes the following diagram commute:
$$\xymatrix{  [\overline{0}, \overline{1}]\times [\overline{0},\overline{1}]  \ar[r]^-{F} \ar[d]_{\rho_{2k+1}}
 & [\overline{0}, \overline{1}]^n  \ar[d]^{\rho_{2k+1}} \\
[0, 1]\times [0,1] \ar[r]_-{f} &   [0, 1]^n.}$$
In particular, the image of each of the four quadrants of $[\overline{0}, \overline{1}]^2$ under $F$ is contained in one of the $4$ (not-necessarily distinct) orthants
$$\big\{ S\big( f(0, 0), 2k+1\big), S\big( f(1, 0), 2k+1\big), S\big( f(0, 1), 2k+1\big), S\big( f(1, 1), 2k+1\big)  \big\}$$
of $[\overline{0}, \overline{1}]^n$.
\end{lemma}

Our proof of \lemref{lem: unit square covering} makes use of the following device.

\begin{definition}[Coordinate-centring function]\label{def:Centring C}
Define the map $C \colon [\overline{0}, \overline{1}] \to [\overline{0}, \overline{1}]$ by
$$C(x) = \begin{cases} x+1 & \overline{0} = k \leq x \leq 2k-1 = \overline{0} + k - 1\\
x & x = \overline{0} + k = 2k \text{ or } x = \overline{0} + k +1 = 2k+1\\
x-1 & \overline{0} + k + 2 = 2k+2 \leq x \leq 3k +1= \overline{1}.\end{cases}
$$
We refer to this map $C$ as the \emph{coordinate-centring function}.
\end{definition}

The coordinate-centring function plays a prominent role in all that follows.  We will develop some of its uses before proving \lemref{lem: unit square covering}.
The idea is that $C$ may be used to progressively move each coordinate of a  point of $[\overline{0}, \overline{1}]^n$  closer to that of a ``central" point, in a certain sense.
Namely, for any $n$, the $n$-cube $[\overline{0}, \overline{1}]^n$ has a \emph{central $2^n$-clique}, which is the unit $n$-cube  $[\overline{0} + k, \overline{0} +k+1]^n$ at  the centre of $[\overline{0}, \overline{1}]^n$. For instance, the central $4$-clique of $[\overline{0}, \overline{1}]^2$ consists of the $4$ points
$$\{ (\overline{0} + k, \overline{0} +k), (\overline{0} + k+1, \overline{0} +k), (\overline{0} + k, \overline{0} +k+1), (\overline{0} + k+1, \overline{0} +k+1) \}.$$
Define a function $c\colon \{k, 3k+1\} \to \{ 2k, 2k+1\}$ by $c(k) = 2k$ (or $c(\overline{0}) = \overline{0} + k$) and $c(3k+1) = 2k+1$ (or $c(\overline{1}) = \overline{0} + k+1$).  Then for each corner of $[\overline{0}, \overline{1}]^n$,
$$(y_1, \ldots, y_n) \text{ with }  \{ y_1, \ldots, y_n\} \subseteq \{ \overline{0}, \overline{1} \},$$
the closest point to that corner in the central clique of $[\overline{0}, \overline{1}]^n$ is $(c(y_1), \ldots,  c(y_n))$.  By iterating the coordinate-centring function, we may obtain the same result: for each coordinate of  the corner point, we have $C^k(y_i) = c(y_i)$.  Indeed, we can parametrize the path in $[\overline{0}, \overline{1}]^n$  from corner to closest central-clique point as
\begin{equation}\label{eq: diagonal of orthant}
\{ (C^s(y_1), \ldots, C^s(y_n)) \mid s = 0, \ldots, k\},
\end{equation}
where we mean $C^0(y_i) = y_i$.
We may divide the $n$-cube $[\overline{0}, \overline{1}]^n$ into $2^n$ sub-cubes, or \emph{orthants} (quadrant if $n=2$) as we will refer to them in the sequel, consisting of products of $n$  intervals
$$I_1 \times \cdots \times I_n,$$
with each interval $I_j$ equal to $[\overline{0}, 2k]$ or $[2k+1, \overline{1}]$.  Then the points \eqref{eq: diagonal of orthant} constitute a diagonal from (outside) corner to opposite (central) corner of one such orthant.

The coordinate-centring function is also useful for describing the other points in each quadrant of $[\overline{0}, \overline{1}]^2$, as well as edges and diagonals of faces in
$[\overline{0}, \overline{1}]^n$.
 In $[\overline{0}, \overline{1}]^2$, from a corner $P = (\overline{v_1}, \overline{v_2})$, with $(v_1, v_2) \in [0, 1]^2$, the parts of the horizontal and vertical edges that leave the corner, and are in the same quadrant of $[\overline{0}, \overline{1}]^2$ as that corner, consist of the points
$$\{ (C^s(\overline{v_1}), \overline{v_2}) \mid s = 0, \ldots, k\} \quad \text{and} \quad  \{ \overline{v_1}, C^t(\overline{v_2})) \mid t = 0, \ldots, k\},$$
respectively.   In fact, we may re-describe the interpolation of \eqref{eq: F on edge} entirely in terms of the coordinate-centring function, as follows.

If $(v_1, v_2) \in [0, 1]^2$, then notice that $(1-v_1, v_2)$ is  the horizontally opposite corner of $[0, 1]^2$  and $(v_1, 1-v_2)$ is  the vertically opposite corner.

\begin{lemma}\label{lem: C parametrize}
With reference to the set-up for \lemref{lem: unit square covering} above, write $f\colon [0, 1]^2 \to [0, 1]^n$ and $F \colon [\overline{0}, \overline{1}]^2 \to [\overline{0}, \overline{1}]^n$ coordinate-wise, as
$$f(v) = \big( f_1(v), \ldots, f_n(v) \big) \quad \text{and} \quad F(x) = \big( F_1(x), \ldots, F_n(x) \big),$$
with $v = (v_1, v_2) \in  [0, 1]^2$ and $x \in [\overline{0}, \overline{1}]^2$.
\begin{itemize}
\item[(A)] Points in the same quadrant of $[\overline{0}, \overline{1}]^2$ as $\overline{v}$, and along the horizontal edge that leaves $\overline{v}$ towards its horizontally opposite corner $\overline{v'} = \overline{(1-v_1, v_2)}$, are given by
$$\overline{v} + s[ (1-v_1, v_2) - v] = \big( C^s(\overline{v_1}), \overline{v_2}\big) \text{ for } s= 0, \ldots, k.$$
\item[(B)] Points in the same quadrant of $[\overline{0}, \overline{1}]^2$ as $\overline{v}$, and along the vertical edge that leaves $\overline{v}$ towards its vertically opposite corner
$\overline{v'} =  \overline{(v_1, 1-v_2)}$, are given by
$$\overline{v} + t[ (v_1, 1-v_2) - v] = \big( \overline{v_1},  C^t(\overline{v_2})\big) \text{ for } t= 0, \ldots, k.$$
\item[(C)] Along the horizontal edge of (A), we may re-write the interpolation of  \eqref{eq: F on edge} coordinate-wise  in the form
$$F_i\big( \overline{v} + s[ (1-v_1, v_2) - v]\big) = F_i\big( C^s(\overline{v_1}), \overline{v_2}\big) = \begin{cases} C^s\big( \overline{f_i(v)} \big) & f_i(v) \not= f_i(1-v_1, v_2)\\
\overline{f_i(v)} & f_i(v) = f_i(1-v_1, v_2)\end{cases}$$
for each coordinate function $i = 1, \ldots, n$ and for each point on the edge $s= 0, \ldots, k$.
\item[(D)] Along the vertical edge of (B), we may re-write the interpolation of  \eqref{eq: F on edge} coordinate-wise  in the form
$$F_i\big( \overline{v} + t[ (v_1, 1-v_2) - v]\big) = F_i\big( \overline{v_1}, C^t(\overline{v_2})\big) = \begin{cases} C^t\big( \overline{f_i(v)} \big) & f_i(v) \not= f_i(v_1, 1-v_2)\\
\overline{f_i(v)} & f_i(v) = f_i(v_1, 1-v_2)\end{cases}$$
for each coordinate function $i = 1, \ldots, n$ and for each point on the edge $t= 0, \ldots, k$.
\end{itemize}
\end{lemma}

\begin{proof}
(A)  With $v = (v_1, v_2) \in [0, 1]^2$, we have $(1-v_1, v_2) - v = (1-2v_1, 0)$, and
$$1-2v_1 =
\begin{cases} +1 & v_1 = 0\\
 -1 & v_1 = 1.\end{cases}
 $$
Meanwhile, for $1 \leq s \leq k$, we have
$$C^s(\overline{v_1}) =
\begin{cases} C^{s-1}(\overline{v_1})+1 & v_1 = 0\\
 C^{s-1}(\overline{v_1})-1 & v_1 = 1.\end{cases}
 $$
 It follows that we have
$$\overline{v} + s[ (1-v_1, v_2) - v] = \big( C^s(\overline{v_1}), \overline{v_2}\big)= \begin{cases} (\overline{v_1} + s, \overline{v_2}) & v_1 = 0\\
(\overline{v_1} - s, \overline{v_2}) & v_1 = 1,\end{cases}$$
for each $s$ with $0 \leq s \leq k$.

(B) Similar reasoning shows that, here, we have
$$\overline{v} + t[ (v_1, 1-v_2) - v] = \big(\overline{v_1}, C^t(\overline{v_2})\big)= \begin{cases} (\overline{v_1}, \overline{v_2} + t) & v_2 = 0\\
(\overline{v_1}, \overline{v_2} - t) & v_2 = 1,\end{cases}$$
for each $t$ with $0 \leq t \leq k$.

(C)  The interpolation \eqref{eq: F on edge} along (the part of) a horizontal edge of  $[\overline{0}, \overline{1}]^2$ that leaves the corner $\overline{v} = \overline{(v_1, v_2)}$, towards its horizontally opposite corner $\overline{v'} =  \overline{(1-v_1, v_2)}$, may be re-written---incorporating (A)---as
$$F\big( C^s(\overline{v_1}), \overline{v_2}\big) = \overline{f(v)} + s[ f(1-v_1, v_2) - f(v)]$$
for $s = 0, \ldots, k$.
Coordinate-wise, we have
$$F_i\big( C^s(\overline{v_1}), \overline{v_2}\big) = \overline{f_i(v)} + s[ f_i(1-v_1, v_2) - f_i(v)]$$
for each $i = 1, \ldots, n$.  Now, on the one hand, we have
$$f_i(1-v_1, v_2) - f_i(v) =
\begin{cases} +1 & f_i(v) \not= f_i(1-v_1, v_2) \text{ and } f_i(v) = 0\\
 -1 & f_i(v) \not= f_i(1-v_1, v_2) \text{ and } f_i(v) = 1\\
 0 & f_i(v) = f_i(1-v_1, v_2).\end{cases}
 $$
On the other hand, for  $1 \leq s \leq k$, we have
$$C^s\big( \overline{f_i(v)} \big) =
\begin{cases} C^{s-1}(\overline{f_i(v)})+1 & f_i(v) = 0\\
 C^{s-1}(\overline{f_i(v)})-1 & f_i(v) = 1.\end{cases}
 $$
It follows that we have
$$\overline{f_i(v)} + s[ f_i(1-v_1, v_2) - f_i(v)] = \begin{cases} C^s\big( \overline{f_i(v)} \big) & f_i(v) \not= f_i(1-v_1, v_2)\\
\overline{f_i(v)} & f_i(v) = f_i(1-v_1, v_2)\end{cases}$$
as asserted.

(D) With $\overline{v} = \overline{(v_1, v_2)}$ and its vertically  opposite corner $\overline{v'} =  \overline{(v_1, 1-v_2)}$, similar steps to those followed in proving (C) result in
$$\overline{f_i(v)} + s[ f_i(v_1,1- v_2) - f_i(v)] = \begin{cases} C^t\big( \overline{f_i(v)} \big) & f_i(v) \not= f_i(v_1, 1-v_2)\\
\overline{f_i(v)} & f_i(v) = f_i(v_1, 1-v_2),\end{cases}$$
and hence the assertion.
\end{proof}

Finally, by way of developing uses of the of coordinate-centring function, we note that, for $v = (v_1, v_2) \in [0,1]^2$, the quadrant of points of $[\overline{0}, \overline{1}]^2$ that contains the corner $\overline{v} = (\overline{v_1}, \overline{v_2})$ may be described as the set of points
$$\{ \big(C^s(\overline{v_1}), C^t(\overline{v_2}\big) \mid 0 \leq s, t \leq k\}.$$
Amongst these points, we may distinguish the outer edges of the quadrant by (A) and (B) of \lemref{lem: C parametrize}, and the diagonal of this quadrant by \eqref{eq: diagonal of orthant} ($n = 2$).  Effectively, the coordinate-centring function provides us with a coordinatization of each quadrant of $[\overline{0}, \overline{1}]^2$.

Having thus prepared the ground thoroughly, we now embark upon our proof:

\begin{proof}[Proof of \lemref{lem: unit square covering}]
The idea is to ``fold" the square $[\overline{0}, \overline{1}]^2$ into the cube $[\overline{0}, \overline{1}]^n$, matching the edges of the square with the edges or diagonals of the cube as specified by the give $F$.  Note however, that $F$ may map different corners to the same corner, and also $F$ will not be an embedding in general.  Furthermore, even when $F$ does end up an embedding, we interpolate using a number of points:  For us, an edge  and any diagonal of a cube have the same ``length," but this is not so geometrically.   So the extension of $F$ to the square will not literally be a fold.

Divide each quadrant of the square $[\overline{0}, \overline{1}]^2$ into two triangles using the diagonals of the square.  The situation is illustrated in (A) of \figref{fig: Lem 5-2 second}. Once again (round) dots--both solid and open---indicate points on which $F$ is already defined.  Squares indicate (interior) points on the diagonals; we have yet to extend $F$ over these points.  We have preserved the (dotted) vertical and horizontal gridlines that appeared in the figures of the proof of \thmref{thm: 2-D subdivision map rectangle} and whose attributes were described there.  Recall that these gridlines do not pass through points, but do separate the square into quadrants, each of which projects to one corner of  $[0, 1]\times [0,1]$ under $\rho_{2k+1}$.
(To pursue the folding analogy a little, the diagonals and these horizontal and vertical gridlines are the folds of a \emph{square base}, or \emph{waterbomb base}, preliminary fold---see, e.g., \cite[p.241]{DE-OR07}.)

As discussed around \eqref{eq: diagonal of orthant}, the segment along that diagonal from a corner of $[\overline{0}, \overline{1}]^2$ to its closest corner of the central $4$-clique---namely, the unique point of the central $4$-clique in the same quadrant as the corner---is a segment of length $k$.  Likewise in $[\overline{0}, \overline{1}]^n$:  Each corner of $[\overline{0}, \overline{1}]^n$ lies in a unique orthant, that also contains a unique corner of the central clique, and the segment from that corner of the $n$-cube to the (closest) corner of central clique that lies in its orthant is also a segment of length $k$.
As part of our extension of $F$, we match each diagonal  segment from corner to central clique in  $[\overline{0}, \overline{1}]^2$ with the segment from corner to central clique in $[\overline{0}, \overline{1}]^n$, in that orthant  determined by the image under $F$ of the corner from $[\overline{0}, \overline{1}]^2$.
\begin{figure}[h!]
\centering
   \begin{subfigure}{0.49\linewidth} \centering
    \includegraphics[trim=280 490 120 130,clip,width=\textwidth]{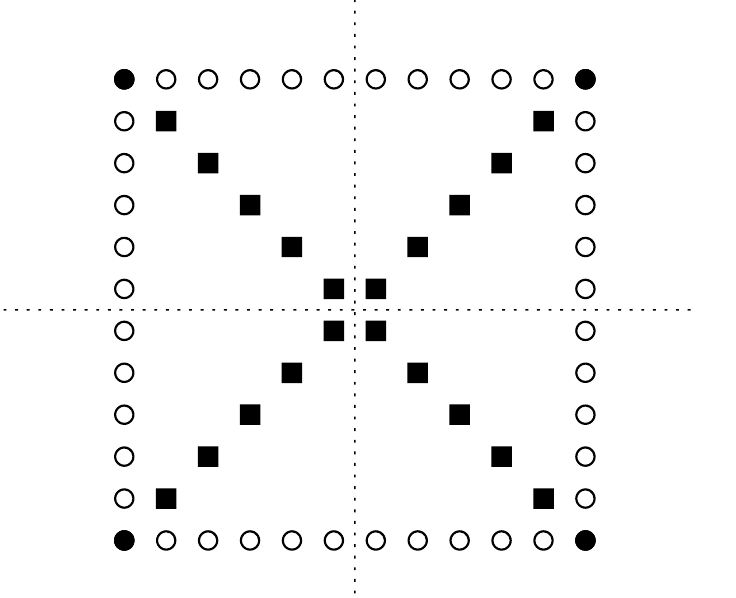}
     \caption{Diagonals, a.k.a.~Segments from Corner to Central Clique}\label{fig:Lem 5-2 secondfigA}
   \end{subfigure}
   \begin{subfigure}{0.49\linewidth} \centering
    \includegraphics[trim=280 490 120 130,clip,width=\textwidth]{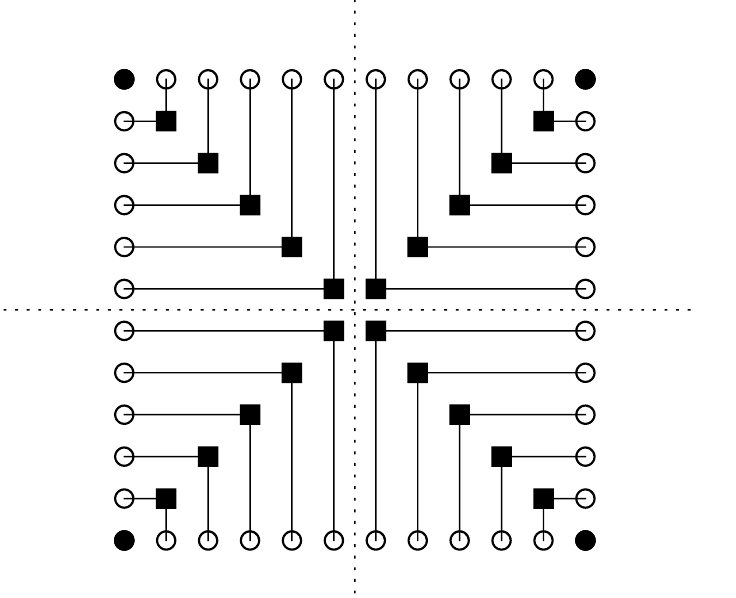}
     \caption{Interpolation Scheme in Each Quadrant}\label{fig:Lem 5-2 secondfigB}
   \end{subfigure}
\caption{Steps in Extension of $F$ (Illustrated with $2k+1 = 11$)} \label{fig: Lem 5-2 second}
\end{figure}
With \eqref{eq: diagonal of orthant} and the notation established in \lemref{lem: C parametrize}, we formulate this as follows.  For $v = (v_1, v_2) \in  [0, 1]^2$ and corresponding corner $\overline{v} = (\overline{v_1}, \overline{v_2}) \in [\overline{0}, \overline{1}]^2$, we define $F$ on the diagonal of the quadrant of $[\overline{0}, \overline{1}]^2$ that contains
$\overline{v}$ as
\begin{equation}\label{eq: F on diagonal}
F\big( C^s(\overline{v_1}), C^s(\overline{v_2}) \big) =  \big( C^s(\overline{f_1(v)}), \ldots, C^s(\overline{f_n(v)}) \big) \text{ for } s = 0, \ldots, k.
\end{equation}
Then, our scheme for completing the extension of $F$ is, in each quadrant of  $[\overline{0}, \overline{1}]^2$, to interpolate the values of $F$ from those on the outer edges of the quadrant to those on the diagonal.  The scheme is illustrated in (B) of \figref{fig: Lem 5-2 second}, with the (solid) lines indicating the lines along which we interpolate.

So, fix a quadrant of $[\overline{0}, \overline{1}]^2$ by choosing $v = (v_1, v_2) \in  [0, 1]^2$, with corresponding corner $\overline{v} = (\overline{v_1}, \overline{v_2}) \in [\overline{0}, \overline{1}]^2$ that determines its quadrant.  Recall that in the formulations of \lemref{lem: C parametrize}, we used the observation that, for
$v = (v_1, v_2) \in [0, 1]^2$, then  $(1-v_1, v_2)$ is  the horizontally opposite corner of $[0, 1]^2$  and $(v_1, 1-v_2)$ is  the vertically opposite corner.  Also, note that we are using coordinate-wise descriptions of $F$ and $f$ in the following.  On the quadrant of $[\overline{0}, \overline{1}]^2$ that contains $\overline{v}$. then, we define
\begin{equation}\label{eq: coordinate-wise interpolant}
F_i\big(  C^s(\overline{v_1}), C^t(\overline{v_2})\big) =
\begin{cases}
C^s\big( \overline{f_i(v)} \big) & t \leq s \text{ and } f_i(v) \not= f_i(1-v_1, v_2)\\
C^t\big( \overline{f_i(v)} \big) & t \leq s \text{ and } f_i(v) = f_i(1-v_1, v_2)\\
C^t\big( \overline{f_i(v)} \big) & s \leq t \text{ and } f_i(v) \not= f_i(v_1, 1-v_2)\\
C^s\big( \overline{f_i(v)} \big) & s \leq t \text{ and } f_i(v) = f_i(v_1, 1-v_2)
\end{cases}
\end{equation}
It is easy to see that this definition achieves the interpolation scheme indicated above: The formula specializes to retrieve formula  \eqref{eq: F on diagonal} on the diagonal ($s=t$) of this quadrant, as well as the (re-formulated versions of the) description of $F$ on the outer edges of the quadrant as in parts (C) and (D) of  \lemref{lem: C parametrize}.  Applying this formula to each quadrant of  $[\overline{0}, \overline{1}]^2$ extends $F$ over the whole square $[\overline{0}, \overline{1}]^2$.
See (A) of \figref{fig: Lem 5-2 third} for an illustration of this last extension of $F$.  Some of the (new) points on which we are defining $F$ at this step are indicated there by stars.  In this figure, we have adopted geographical, ``points of  the compass"  terminology to identify the various quadrants and the triangles within them.
\begin{figure}[h!]
\centering
   \begin{subfigure}{0.45\linewidth} \centering
    \includegraphics[trim=220 420 170 160,clip,width=\textwidth]{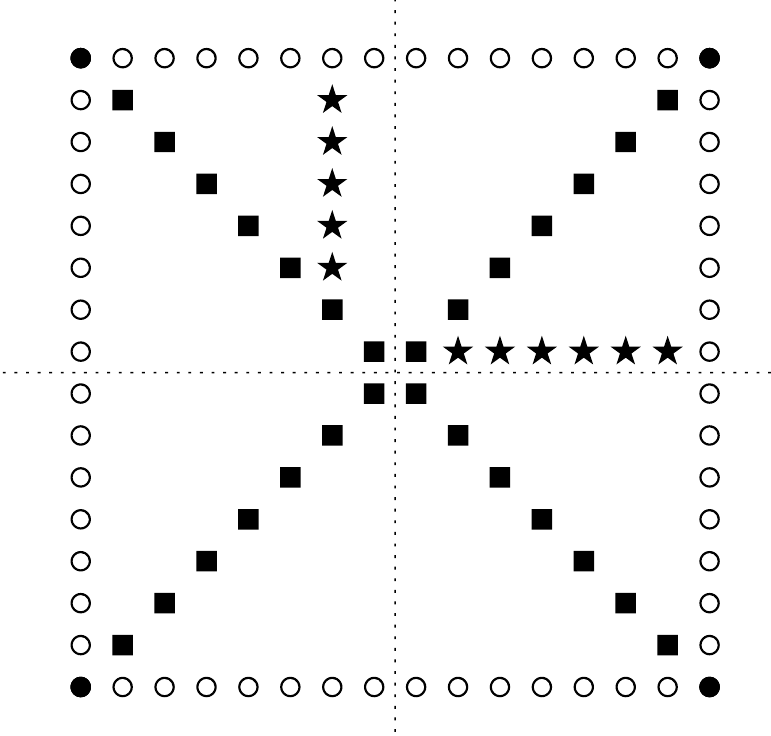}
     \caption{Interpolating from Outer Edge  to Diagonal: $s = 6$ and $t=0, \ldots, 6$ in $NNW$ and $s=7$  and $t=0, \ldots, 7$ in $ENE$}\label{fig:Lem 5-2 thirdfigA}
   \end{subfigure}
   \begin{subfigure}{0.45\linewidth} \centering
    \includegraphics[trim=220 420 170 160,clip,width=\textwidth]{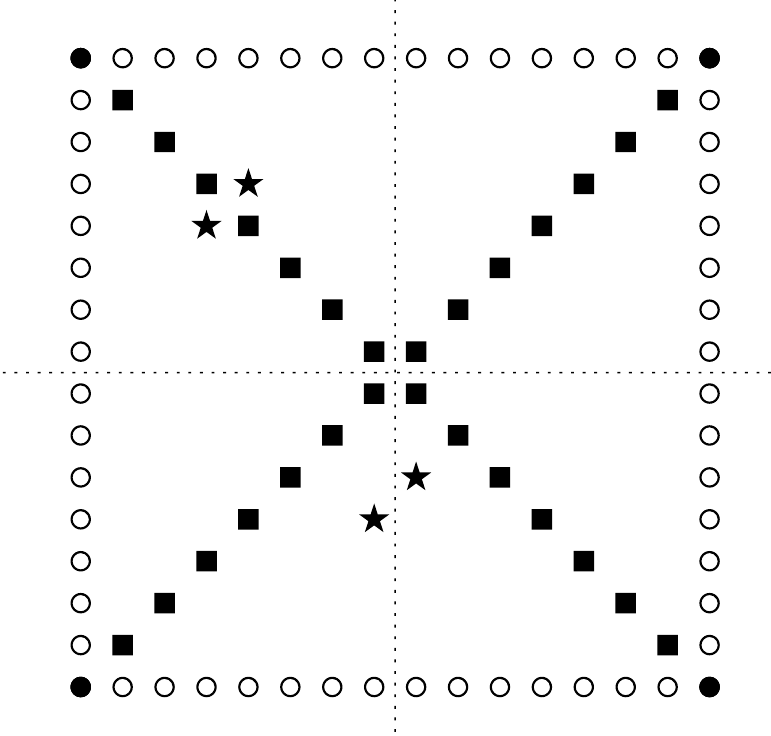}
     \caption{Adjacent Points in Same Quadrant, Different Triangles ($(s, t) = (4, 3)$ in $NNW$ and $WNW$),  and in Different Quadrants ($(s, t) = (7, 4)$ in $SW$ and $(s, t) = (7, 5)$ in $SE$) }\label{fig:Lem 5-2 thirdfigB}
   \end{subfigure}
\caption{Final Step in Extension of $F$ (Illustrated with $2k+1 = 15$)} \label{fig: Lem 5-2 third}
\end{figure}

It remains to check continuity of the extended $F$.  To this end, suppose we have two adjacent points $(p, q) \sim (p', q') \in [\overline{0}, \overline{1}]^2$.  We must confirm that $F(p, q) \sim F(p', q') \in [\overline{0}, \overline{1}]^n$.  We divide the possibilities into three cases: (i) both  $(p, q)$ and $(p', q')$ lie in a single triangle of one quadrant  (including the boundaries of said triangle); (ii) both  $(p, q)$ and $(p', q')$ lie in a single quadrant, but in different triangles of that quadrant;  (iii)  $(p, q)$ and $(p', q')$ lie in different quadrants.  Cases (ii) and (iii) are illustrated in (B) of \figref{fig: Lem 5-2 third}, in which the pairs of adjacent points are represented by stars.

\subsection{Case (i): Same Triangle}
Write the corner of the quadrant as $(\overline{v_1}, \overline{v_2})$, for $v = (v_1, v_2)$ a corner of $[0, 1]^2$.   With points in the quadrant given as $\big(  C^s(\overline{v_1}), C^t(\overline{v_2})\big)$ suppose that our points are in the triangle in which $t \leq s$.  We may write the two points as
$$(p, q) = \big(  C^s(\overline{v_1}), C^t(\overline{v_2})\big) \quad \text{and} \quad    (p', q') = \big(  C^{s'}(\overline{v_1}), C^{t'}(\overline{v_2})\big),$$
with $0 \leq t \leq s \leq k$, $0 \leq t' \leq s' \leq k$, and, because $(p, q) \sim (p', q')$, we must have $|s' - s| \leq 1$ and $|t' - t| \leq 1$.  From our  coordinate-wise definition of $F$ in \eqref{eq: coordinate-wise interpolant}, if $f_i(v) \not= f_i(1-v_1, v_2)$, then we have
$$F_i(p, q) = C^s\big( \overline{f_i(v)} \big) \quad \text{and} \quad F_i(p', q') = C^{s'}\big( \overline{f_i(v)} \big),$$
which differ by at most $1$ from each other, since we have $|s' - s| \leq 1$.  But if  $f_i(v) = f_i(1-v_1, v_2)$, then we have
$$F_i(p, q) = C^t\big( \overline{f_i(v)} \big) \quad \text{and} \quad F_i(p', q') = C^{t'}\big( \overline{f_i(v)} \big),$$
which again differ by at most $1$ from each other, since we also have $|t' - t| \leq 1$.  Each coordinate of $F(p, q)$ and $F(p', q')$ differs by at most $1$, meaning that $F(p, q) \sim F(p', q')$.   If our points are in the triangle in which $s \leq t$, then a similar argument using the appropriate cases of \eqref{eq: coordinate-wise interpolant} arrives at the same conclusion.    This shows that $F$ preserves adjacencies in Case (i).

\subsection{Case (ii): Same Quadrant, Different Triangles}
A typical situation in this case is that illustrated in the $NW$ quadrant of (B) of \figref{fig: Lem 5-2 third} (adjacent points represented by stars).  Unless both $(p, q)$ and $(p', q')$ are in one triangle, which would place us back in Case (i), we must have $|p'-p| = 1$ and $|q' - q| = 1$ with $(p, q)$ and $(p', q')$ lying on either side of a diagonal.
Points in the quadrant are $\big(  C^s(\overline{v_1}), C^t(\overline{v_2})\big)$, with  $(\overline{v_1}, \overline{v_2})$ the corner of the quadrant and the diagonal consisting of those points with $0 \leq s=t \leq k$.  WLOG, suppose we have $(p, q) = \big(  C^s(\overline{v_1}), C^t(\overline{v_2})\big)$ with $t < s$ and $(p', q') = \big(  C^{s'}(\overline{v_1}), C^{t'}(\overline{v_2})\big)$ with $s'< t'$. Then, for some $s$ with $1 \leq s \leq k-1$, we must have $(p, q) = \big( C^s(x_1), C^{s-1}(x_2)\big)$ and $(p', q')= \big( C^{s-1}(x_1), C^{s}(x_2)\big)$.  Then for each coordinate, we have
$$F_i(p, q) =  \begin{cases}
 C^s\big( \overline{f_i(v)} \big) &  f_i(v) \not= f_i(1-v_1, v_2)\\
 C^{s-1}\big( \overline{f_i(v)} \big) & f_i(v) = f_i(1-v_1, v_2)\end{cases}$$
and
$$F_i(p', q') =  \begin{cases}
C^{s-1}\big( \overline{f_i(v)} \big)&  f_i(v) \not= f_i(v_1, 1-v_2)\\
  C^s\big( \overline{f_i(v)} \big)& f_i(v) = f_i(v_1, 1-v_2) .\end{cases}$$
The possible values here either agree or they differ by $1$ since each application of the coordinate-centring function $C$ increases or decreases the input by $1$.  Either way, each coordinate of $F(p, q)$ and $F(p', q')$ differs by at most $1$, meaning that $F(p, q) \sim F(p', q')$ and  $F$ preserves adjacencies in Case (ii) also.

\subsection{Case (iii): Different Quadrants}
Here, a typical situation  is that illustrated in the $SE$ and $SW$ quadrants of (B) of \figref{fig: Lem 5-2 third} (again, adjacent points represented by stars).  Note that we have defined $F$ so that the central clique of  $[\overline{0}, \overline{1}]^2$ is mapped to a subset of the central clique of $[\overline{0}, \overline{1}]^n$ in which, tautologically, every point is adjacent to every other.  Therefore, in case (iii), we need not consider situations in which the two adjacent points are in diagonally adjacent quadrants, which would force both points to be in the central clique of $[\overline{0}, \overline{1}]^2$.     Suppose the points are in horizontally adjacent quadrants, with corners $(\overline{v_1}, \overline{v_2})$ and $(\overline{1-v_1}, \overline{v_2})$.  WLOG, suppose that $(p, q) = \big(  C^s(\overline{v_1}), C^t(\overline{v_2})\big)$ and $(p', q') = \big(  C^s(\overline{1-v_1}), C^t(\overline{v_2})\big)$.  Because they are adjacent, we must have $s=k=s'$, with $|t' - t| \leq 1$.  Then we have
$$F_i(p, q) =  \begin{cases}
 C^k\big( \overline{f_i(v)} \big) &  f_i(v) \not= f_i(1-v_1, v_2)\\
 C^{t}\big( \overline{f_i(v)} \big) & f_i(v) = f_i(1-v_1, v_2)\end{cases}$$
and
$$F_i(p', q') =  \begin{cases}
C^{k}\big( \overline{f_i(1-v_1, v_2)} \big)&  f_i(1-v_1, v_2) \not= f_i(v)\\
C^{t'}\big( \overline{f_i(1-v_1, v_2)} \big)& f_i(1-v_1, v_2) = f_i(v) .\end{cases}$$
 Note, here, that we are using symmetry when applying the formulas of \eqref{eq: coordinate-wise interpolant} to two different quadrants: since $(1-v_1, v_2)$ is the corner of $[0, 1]^2$ horizontally opposite, $v$, so too is $v$ the corner of $[0, 1]^2$ horizontally opposite $(1-v_1, v_2)$.
 If $f_i(v) = f_i(1-v_1, v_2)$, then $F_i(p, q) = C^{t}\big( \overline{f_i(v)} \big)$ and $F_i(p', q') = C^{t'}\big( \overline{f_i(v)} \big)$  differ by at most one, since we have $|t' - t| \leq 1$.
 If $f_i(v) \not= f_i(1-v_1, v_2)$, then we compute, as in the proof of part (A) of  \lemref{lem: C parametrize}, that
$$C^k\big( \overline{f_i(v)} \big)  =  \begin{cases}
 2k&  f_i(v) =0\\
2k+1& f_i(v) = 1\end{cases}$$
and
$$C^k\big( \overline{f_i(1-v_1, v_2)} \big)  =  \begin{cases}
 2k&  f_i(1-v_1, v_2) =0\\
2k+1& f_i(1-v_1, v_2) = 1,\end{cases}
$$
whence we have $|F_i(p, q) - F_i(p', q')| = 1$ (both $f_i(v)$ and $f_i(1-v_1, v_2)$ must be either $0$ or $1$, remember). Either way, each coordinate of $F(p, q)$ and $F(p', q')$ differs by at most $1$, and we have $F(p, q) \sim F(p', q')$. If the two adjacent points are in vertically adjacent quadrants of $[\overline{0}, \overline{1}]^2$, a similar argument arrives at the same conclusion.  This completes the check of continuity  in Case (iii) and with it, the proof.
\end{proof}

Now we extend  \thmref{thm: 2-D subdivision map rectangle} to the case in which the domain is an arbitrary 2D digital image.

\begin{theorem}\label{thm: 2-D subdivision map}
Suppose we are given a map $f \colon X \to Y$ of digital images $X \subseteq \Z^2$ and $Y \subseteq \Z^n$.  For any $k \geq 1$,  there is  a map $\widehat{f} \colon S(X, 2k+1) \to S(Y, 2k+1)$ that makes the following diagram commute:
$$\xymatrix{ S(X, 2k+1) \ar[d]_{\rho_{2k+1}} \ar[r]^-{\widehat{f}} & S(Y, 2k+1) \ar[d]^{\rho_{2k+1}}\\
X \ar[r]_-{f} & Y}$$
\end{theorem}

\begin{proof}
We use the notation established in previous results without comment.
Begin by defining $\widehat{f}$ on centres, as
\begin{equation}\label{eq: general centres}
\widehat{f}(\overline{x}) = \overline{f(x)}
\end{equation}
for each $x \in X$.  Then, in each $S(x, 2k+1) \subseteq S(X, 2k+1)$, extend $\widehat{f}$ horizontally and vertically from the centre to each edge in one of two ways. For each
\emph{vertical or horizontal} neighbour $x \sim x'$ in $X$, interpolate the values of $\widehat{f}$ along the vertical or horizontal segment joining  $\overline{x}$ and $\overline{x'}$ in $S(X, 2k+1)$. Where $x$  is missing one or more of its potential horizontal or vertical neighbours from $X$, extend $\widehat{f}$ as a constant from the centre out to that edge of  $S(x, 2k+1)$.
In terms of a formula, suppose $x \sim_{\Z^2} x'$ are vertical or horizontal neighbours in $\Z^2$, so that $x' - x$ has one coordinate $0$ and the other $\pm1$.    Then we define, for each $t  = 0, \ldots, k$,
\begin{equation}\label{eq: general between centres}
\widehat{f} \big( \overline{x} + t (x' - x)\big) = \begin{cases} \overline{ f(x)} + t[ f(x') - f(x)] & x' \in X\\ \overline{ f(x)}  & x' \not\in X.\end{cases}
\end{equation}
Notice that, in case $x' \in X$, we could equally well define
$$
\widehat{f} \big( \overline{x} + t (x' - x)\big) = \overline{ f(x)} + t[ f(x') - f(x)], \text{ for } t  = 0, \ldots, 2k+1
$$
to give values for $\widehat{f}$ on the segment in $S(X, 2k+1)$ that joins $\overline{x}$ and  $\overline{x'}$ (including the endpoints), and this gives the same values as \eqref{eq: general between centres} on the relevant points of $S(x, 2k+1)$ and $S(x', 2k+1)$.  The case of \eqref{eq: general between centres} in which $x' \in X$ is how we proceeded in \thmref{thm: 2-D subdivision map rectangle}, when vertical and  horizontal neighbours were always present.  Notice also that, at this point, we do not interpolate in this way between diagonal neighbours of $X$, such as $(x_1, x_2)$ and $(x_1+1, x_2+1)$.   See \figref{fig:Th5-4Fig1}  for an illustration of the progress so far, for the case in which $X$ consists of the $5$ points $X = \{ (0, 0), (2, 0), (1, 1),  (2, 1), (2,2)\}$.  As in the illustrations through the proof of \thmref{thm: 2-D subdivision map rectangle},  dots (open or closed)  represent the points on which $\widehat{f}$ has been defined so far. Centres are represented as solid dots; we define $\widehat{f}$ on these points by \eqref{eq: general centres}.  Open dots represent points on which we extend $\widehat{f}$ by \eqref{eq: general between centres}.  As in the proof of \thmref{thm: 2-D subdivision map rectangle}, it remains to extend the definition of $\widehat{f}$ to the points ``outside the centres" and to those in regions ``surrounded by centres." The difference here, though, is that with a non-rectangular $X$, we have a variety of behaviour to consider under each of these titles.
\begin{figure}[h!]
\centering
\includegraphics[trim=130 320 130 120,clip,width=\textwidth]{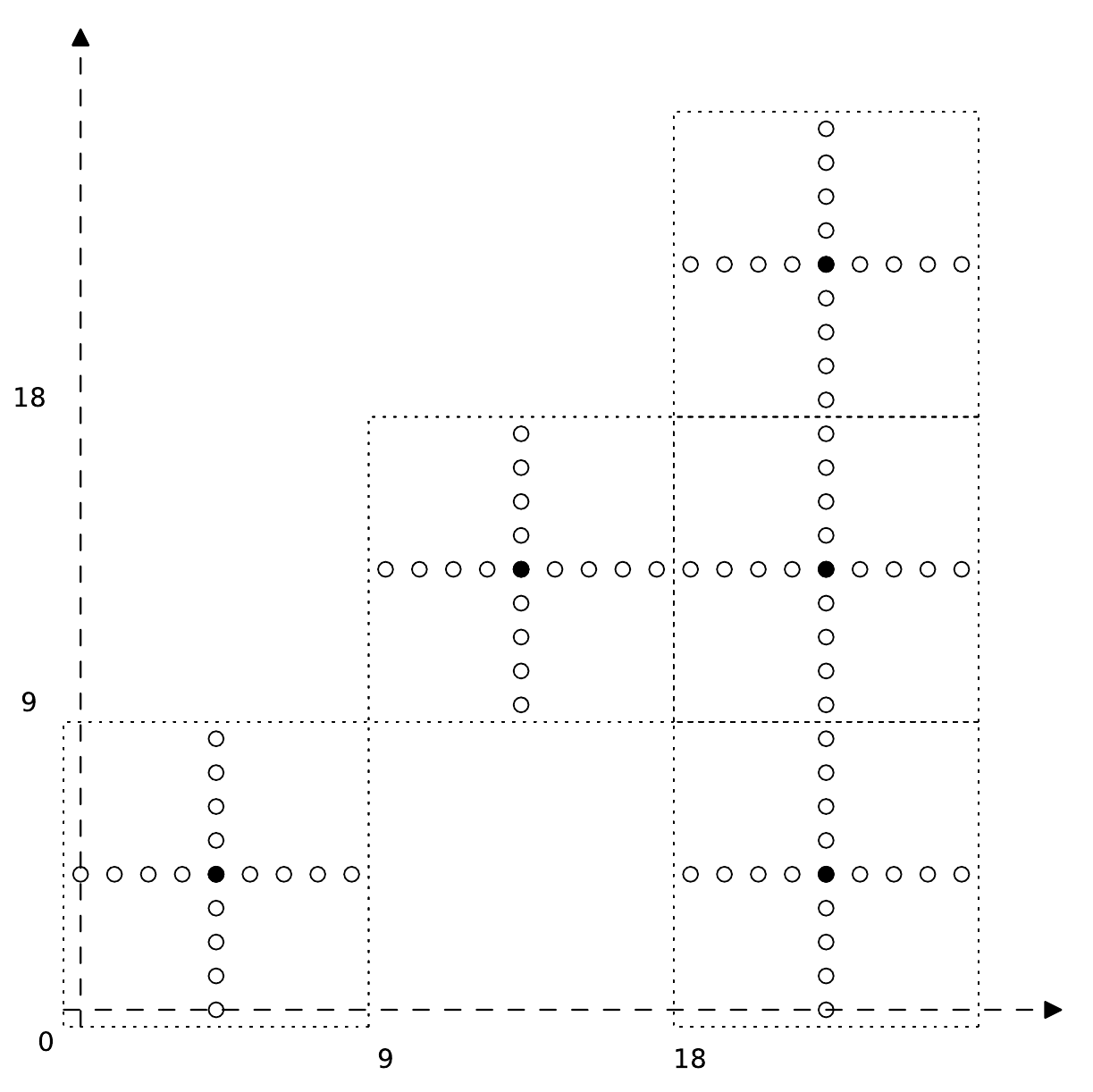}
\caption{  $\widehat{f}$ defined on centres and extended horizontally and vertically within each $S(x, 2k+1)$ by \eqref{eq: general between centres}.  Illustrated with $X = \{ (0, 0), (2, 0), (1, 1),  (2, 1), (2,2)\}$ and  $2k+1 = 9$.}\label{fig:Th5-4Fig1}
\end{figure}

We proceed as follows.   Since $X$ is finite, we may pick some rectangle that contains it.  Suppose we have $X \subseteq [M_1, M_2] \times [N_1, N_2]$ for suitable $M_1, M_2$ and $N_1, N_2$.  Then the squares
$$[\overline{i}, \overline{i+1}] \times [\overline{j}, \overline{j+1}] \text{ for } M_1-1 \leq i \leq M_2+1, N_1-1 \leq j \leq N_2+1$$
cover the whole of $S(X, 2k+1)$.  Some of these squares  may not include any points of  $S(X, 2k+1)$.  But where  $[\overline{i}, \overline{i+1}] \times [\overline{j}, \overline{j+1}] \cap S(X, 2k+1)$ is non-empty, we have already defined $\widehat{f}$ on the parts of the boundary
$$\partial\big( [\overline{i}, \overline{i+1}] \times [\overline{j}, \overline{j+1}] \big) \cap S(X, 2k+1)$$
that belong to $S(X, 2k+1)$, and we use the ideas of  \thmref{thm: 2-D subdivision map rectangle} to extend $\widehat{f}$ over all points of $S(X, 2k+1)$ included in this square.  The idea is that these squares act as ``cookie cutters," to divide $S(X, 2k+1)$, into various sub-regions of the squares, over which we may extend $\widehat{f}$ independently of each other.  This latter observation holds for the same reason it held in the proof of \thmref{thm: 2-D subdivision map rectangle}: any two points adjacent in $S(X, 2k+1)$ must both lie in a single
``cut-out" region
$$\big( [\overline{i}, \overline{i+1}] \times [\overline{j}, \overline{j+1}] \big) \cap S(X, 2k+1) \subseteq S(X, 2k+1)$$
for some $i, j$.  Thus, if we can extend $\widehat{f}$ over each of these pieces separately, we already have $\widehat{f}$ well-defined on their overlaps, and so we may assemble the piecewise-defined map into a global, continuous $\widehat{f}$ on the whole of $S(X, 2k+1)$.  In \figref{fig:Th5-4Fig2}, we have illustrated the idea.
\begin{figure}[h!]
\centering
\includegraphics[trim=130 320 130 120,clip,width=\textwidth]{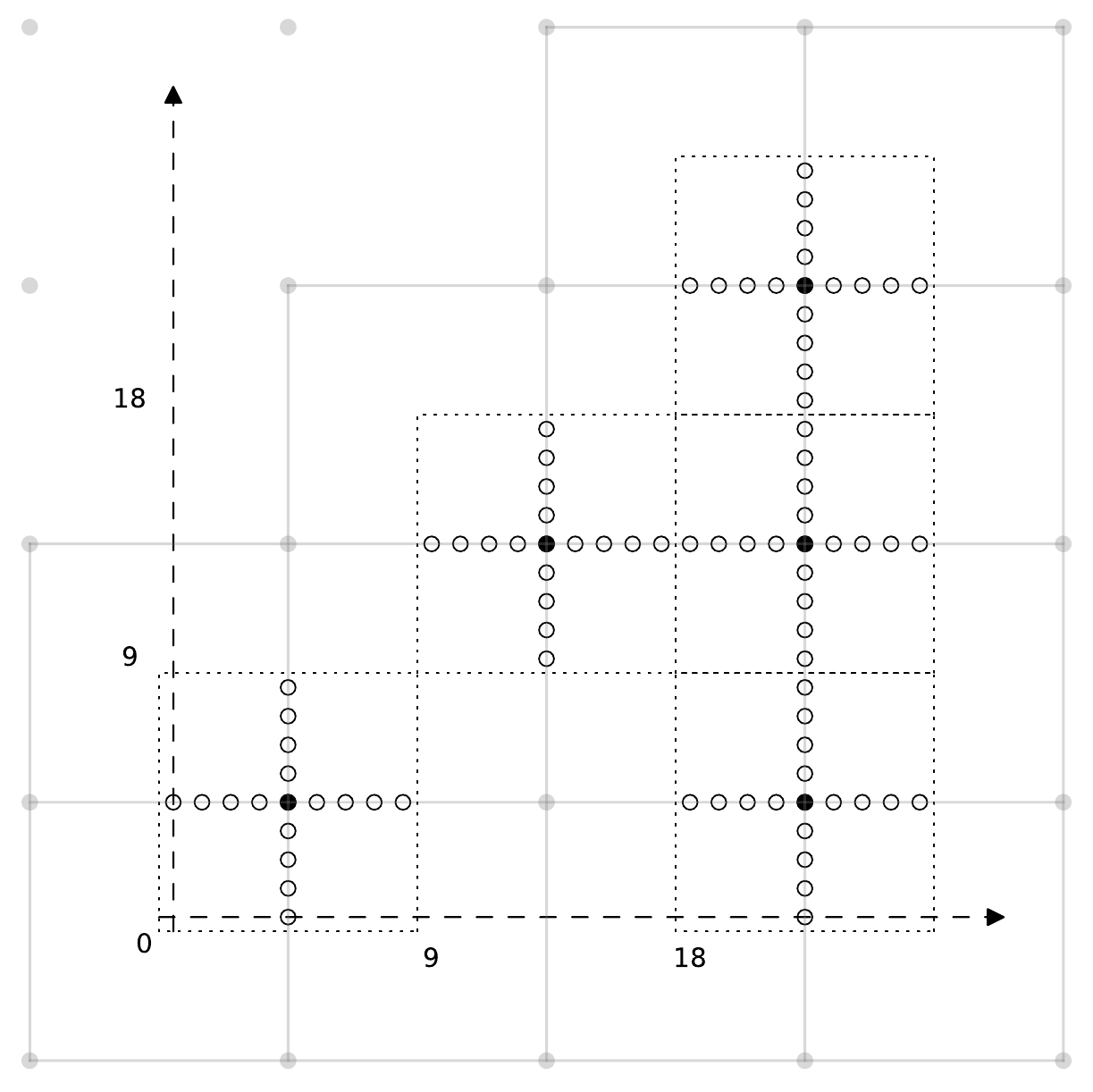}
\caption{  $S(X, 2k+1)$ covered by a rectangle $R = [\overline{M-1}, \overline{M+1}] \times [\overline{N-1}, \overline{N+1}]$.  Centers of $R$ and squares of $R$ for which $I_{i, j} \not= \emptyset$ are illustrated with $2k+1 = 9$.}\label{fig:Th5-4Fig2}
\end{figure}
In the figure, suppose again that $X$ consists of the $5$ points
$$X = \{ (0, 0), (2, 0), (1, 1),  (2, 1), (2,2)\},$$
and that \figref{fig:Th5-4Fig1} represents $\widehat{f}$ defined on the centers of $S(X, 2k+1)$ then extended  by \eqref{eq: general between centres} to the verticals and horizontals of $S(x, 2k+1)$ through each center.  Then in  \figref{fig:Th5-4Fig2}, we have included $S(X, 2k+1)$ in the rectangle $[\overline{-1}, \overline{3}]\times [\overline{-1}, \overline{3}]$ and added (in grey) the centers $(\overline{i}, \overline{j})$ from $S([-1, 3]\times [-1, 3], 2k+1)$.  Also, we have indicated those squares (bounded by grey edges) $[\overline{i}, \overline{i+1}] \times [\overline{j}, \overline{j+1}]$ for which the intersection
$$I_{i, j}:= \big( [\overline{i}, \overline{i+1}] \times [\overline{j}, \overline{j+1}] \big) \cap S(X, 2k+1) \subseteq S(X, 2k+1)$$
is non-empty.  These intersections, generally, consist of the union of any combination of the four quadrants of $[\overline{i}, \overline{i+1}] \times [\overline{j}, \overline{j+1}]$ that we encountered  in the proof of \lemref{lem: unit square covering}. Two cases are illustrated in \figref{fig: Th5-4 third}.
\begin{figure}[h!]
\centering
   \begin{subfigure}{0.49\linewidth} \centering
    \includegraphics[trim=180 470 300 190,clip,width=\textwidth]{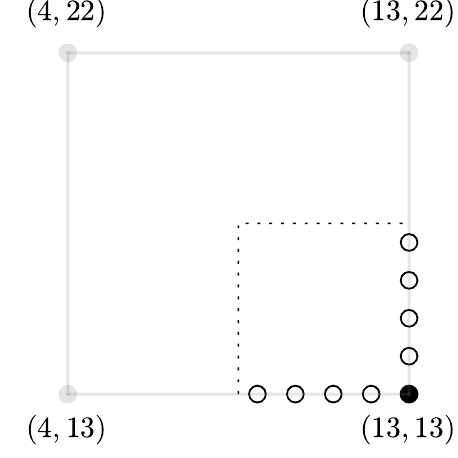}
     \caption{$I_{0, 1}$ gives the  $SE$ quadrant }
   \end{subfigure}
   \begin{subfigure}{0.49\linewidth} \centering
    \includegraphics[trim=180 470 300 190,clip,width=\textwidth]{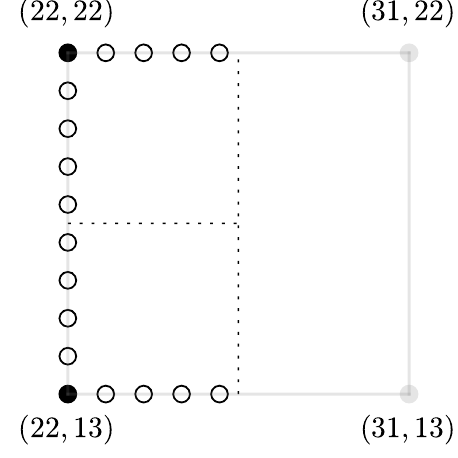}
     \caption{$I_{2, 1}$ gives vertically adjacent \\
     \hbox{\hskip0.25truein}  $NW$ and $SW$ quadrants}
   \end{subfigure}
\caption{One or two quadrants in $I_{i, j}$ (Illustrated with $[\overline{-1}, \overline{3}]^2$ containing $S(X, 2k+1)$ with $2k+1 = 9$)} \label{fig: Th5-4 third}
\end{figure}
Now it is a fact that, although one or more quadrants may be absent from $I_{i, j}$, the same methods as used in \lemref{lem: unit square covering} may be used to extend $\widehat{f}$ over $I_{i, j}$.

  In all cases, we use the device of the proof of \lemref{lem: unit square covering} to reduce the extension over $I_{i, j}$ to one of extending over the comparable parts of $[\overline{0}, \overline{1}]^2$: translate in $\Z^2$ the square $[\overline{i}, \overline{i+1}] \times [\overline{j}, \overline{j+1}]$ and $I_{i, j}$ within it to $[\overline{0}, \overline{1}]^2$ and the corresponding union of quadrants of $[\overline{0}, \overline{1}]^2$; translate some $n$-cube that contains the images under $\widehat{f}$  of all corners of $I_{i, j}$ to the $n$-cube $[\overline{0}, \overline{1}]^n$ in $\Z^n$; translate an extension over the suitable quadrants of $[\overline{0}, \overline{1}]^2$ to obtain an extension over $I_{i, j}$.
\end{proof}

We give the more general version of \lemref{lem: unit square covering} used in the above proof.  Suppose we have a (non-empty) subset
$$V \subseteq [0, 1]^2 = \{ (0, 0), (1, 0), (0, 1), (1,1) \}$$
and a map $f \colon V \to [0, 1]^n$. Those parts of the boundary of $[\overline{0}, \overline{1}]^2 \subseteq S([0, 1]^2, 2k+1)$ that contain points of $S(V, 2k+1)$, namely
$$\partial \big( [\overline{0}, \overline{1}]^2 \big) \cap S(V, 2k+1),$$
consist of the points
$$ \{ (C^s(\overline{v_1}), \overline{v_2}) \mid  (v_1, v_2) \in V, s = 0, \dots, k \} \cup  \{ \big(\overline{v_1}, C^t(\overline{v_2})\big) \mid  (v_1, v_2) \in V, t = 0, \dots, k \},$$
where $C$ denotes the coordinate-centring function of \defref{def:Centring C} used in the proof of \lemref{lem: unit square covering}.
Note that, for any $(v_1, v_2) \in [0, 1]^2$, its horizontal neighbour in $[0, 1]^2$ is $(1-v_1, v_2)$  and its vertical neighbour in $[0, 1]^2$ is $(v_1, 1-v_2)$.  Now suppose that we are given  a partial covering of $f$
$$F \colon \partial \big( [\overline{0}, \overline{1}]^2 \big) \cap S(V, 2k+1) \to [\overline{0}, \overline{1}]^n$$
that satisfies $\rho_{2k+1}\circ F = f\circ\rho_{2k+1}\colon \partial \big( [\overline{0}, \overline{1}]^2 \big) \cap S(V, 2k+1) \to [0, 1]^n$ and is defined on boundary points, for each $v = (v_1, v_2) \in V$ and $s, t = 0, \dots, k$ as
$$F\big( C^s(\overline{v_1}), \overline{v_2} \big) = \begin{cases} \overline{ f(v) } + s[ f(1-v_1, v_2) - f(v)] & (1-v_1, v_2) \in V\\
 \overline{ f(v) } & (1-v_1, v_2) \not\in V,\end{cases}$$
and
$$F\big( \overline{v_1}, C^t(\overline{v_2}) \big) = \begin{cases} \overline{ f(v) } + t[ f(v_1, 1-v_2) - f(v)] & (v_1, 1-v_2) \in V\\
 \overline{ f(v) } & (v_1, 1-v_2) \not\in V,\end{cases}$$

\begin{proposition}\label{prop: unit square covering general}
With the above notation, the map
$$F\colon \partial \big( [\overline{0}, \overline{1}]^2 \big) \cap S(V, 2k+1) \to [\overline{0}, \overline{1}]^n \subseteq \Z^n$$
may be extended in a canonical way to a continuous map
$F\colon [\overline{0}, \overline{1}]^2  \cap S(V, 2k+1) \to [\overline{0}, \overline{1}]^2$
that makes the following diagram commute:
$$\xymatrix{ [\overline{0}, \overline{1}]^2  \cap S(V, 2k+1)  \ar[r]^-{F} \ar[d]_{\rho_{2k+1}}
 & [\overline{0}, \overline{1}]^n  \ar[d]^{\rho_{2k+1}} \\
V \ar[r]_-{f} &   [0, 1]^n.}$$
In particular, for each $v \in V$, the extended $F$ satisfies
$$F\big(  [\overline{0}, \overline{1}]^2  \cap S(v, 2k+1) \big) \subseteq  [\overline{0}, \overline{1}]^n \cap S(f(v), 2k+1).$$
\end{proposition}

\begin{proof}
For each $v \in V$, write
$$I_v = [\overline{0}, \overline{1}]^2  \cap S(v, 2k+1)$$
for the corresponding quadrant of $[\overline{0}, \overline{1}]^2  \cap S(V, 2k+1)$ over which we wish to extend $F$.  If $v = (v_1, v_2)$, then $I_v$ has corner $\overline{v} =  (\overline{v_1}, \overline{v_2})$ and consists of the points
$$ \{ \big( C^s(\overline{v_1}), C^t(\overline{v_2}) \big) \mid  s, t = 0, \dots, k \}.$$
As in \lemref{lem: C parametrize}, we may re-write the given $F$ coordinate-wise on the outer edges of $I_v$ as
\begin{equation}\label{eq: outer edge horizontal}
F_i\big(  C^s(\overline{v_1}), \overline{v_2}\big) =
\begin{cases}
C^s\big( \overline{f_i(v)} \big) & (1-v_1, v_2) \in V \text{ and } f_i(v) \not= f_i(1-v_1, v_2)\\
\overline{f_i(v)} & (1-v_1, v_2) \not\in V \text{ or } f_i(v) = f_i(1-v_1, v_2),
\end{cases}
\end{equation}
and
\begin{equation}\label{eq: outer edge vertical}
F_i\big( \overline{v_1},  C^t(\overline{v_2}) \big) =
\begin{cases}
C^t\big( \overline{f_i(v)} \big) & (v_1, 1-v_2) \in V \text{ and } f_i(v) \not= f_i(v_1, 1-v_2)\\
\overline{f_i(v)} & (v_1, 1-v_2) \not\in V \text{ or } f_i(v) = f_i(v_1, 1-v_2).
\end{cases}
\end{equation}
Then, we interpolate these values over the quadrant $I_v$ using the same scheme as we used to write \eqref{eq: coordinate-wise interpolant} in the proof of
\lemref{lem: unit square covering}.  This leads to the following coordinate-wise definition of $F$ on  $I_v$:

\bigskip

$F_i\big(  C^s(\overline{v_1}), C^t(\overline{v_2})\big) = $

\begin{equation}\label{eq: general coordinate-wise interpolant}
\begin{cases}
C^s\big( \overline{f_i(v)} \big) & t \leq s \text{ and } f_i(v) \not= f_i(1-v_1, v_2)\\
C^t\big( \overline{f_i(v)} \big) & t \leq s \text{ and either } (1-v_1, v_2) \not\in V \text{ or } f_i(v) = f_i(1-v_1, v_2)\\
C^t\big( \overline{f_i(v)} \big) & s \leq t \text{ and } f_i(v) \not= f_i(v_1, 1-v_2)\\
C^s\big( \overline{f_i(v)} \big) & s \leq t \text{ and either } (v_1, 1-v_2) \not\in V \text{ or } f_i(v) = f_i(v_1, 1-v_2)
\end{cases}
\end{equation}
Notice that, on the diagonal of $I_v$, all cases agree and specialize to define
$$F_i\big(  C^s(\overline{v_1}), C^s(\overline{v_2})\big) = C^s\big( \overline{f_i(v)} \big), \text{ for } s = 0, \ldots, k.$$
This formulation applies to extend $F$ over any non-empty quadrant of $[\overline{0}, \overline{1}]^2  \cap S(V, 2k+1)$.  If $V = [0, 1]^2$, then it agrees with the extension of $F$ in  \lemref{lem: unit square covering}.

To check continuity of the extension in case $V \not= [0, 1]^2$,  we argue exactly as we did to check continuity in the proof of \lemref{lem: unit square covering}.  The three cases to consider are the same.  We only need be careful that the extra conditionals of  \eqref{eq: general coordinate-wise interpolant} do not introduce extra possibilities (which they do not).
 \end{proof}

We assert that \thmref{thm: 2-D subdivision map rectangle} and \thmref{thm: 2-D subdivision map} may be extended to even subdivisions as well.  But---as we remarked in
\remref{rem: even subdivisions 1D}, with respect to extending \thmref{thm: path odd subdivision map} to even subdivisions---doing so involves adapting the constructions and arguments so as to replace centres with central cliques.  In fact, for our purposes thus far, it has been sufficient to use \thmref{thm: path odd subdivision map}  and \thmref{thm: 2-D subdivision map}  as we have them, for odd subdivisions only.   If, for some reason it were necessary to involve even subdivisions, then the partial projections can often be used, as we used them in \corref{cor: path odd or even subdivision map}, to obtain covers of even subdivisions using the existence of covers for odd subdivisions.

Nonetheless, for the sake of completeness, we state a result here so as to have a statement of the fact that a covering map exists independently of the parity of $k$.   Just as was the case for \corref{cor: path odd or even subdivision map} vis-{\`a}-vis \thmref{thm: path odd subdivision map},
 the conclusion here for the case in which $k$ is odd is actually weaker than that of \thmref{thm: 2-D subdivision map}.

\begin{corollary}\label{cor: 2D odd or even subdivision map}
Suppose we are given a map $f\colon X \to Y$ with $X \subseteq \Z^2$ a 2D digital image and $Y \subseteq \Z^n$ any digital image.  For any  $k \geq 2$,  there is  a map of subdivisions
$$F\colon S(X, k+1)   \to S(Y, k)$$
that covers the given map, in the sense that the following diagram commutes:
$$\xymatrix{ S(X, k+1) \ar[d]_{\rho_{k+1}} \ar[r]^-{F} & S(Y, k) \ar[d]^{\rho_{k}}\\
X \ar[r]_-{f} & Y}$$
\end{corollary}

\begin{proof}
Suppose that $k$ is even.  Pre-compose $\rho_{k}\colon S(Y, k) \to Y$ with the partial projection $\rho^c_{k+1}\colon S(Y, k+1)  \to S(Y, k)$ of \defref{def: rho^c}. Then, as in  \corref{cor: factor rho}, we have $\rho_{k+1} = \rho_{k} \circ \rho^c_{k+1}\colon S(Y, k+1) \to Y$ and \thmref{thm: 2-D subdivision map} provides a filler for the diagram
$$\xymatrix{ S(X, k+1) \ar[d]_{\rho_{k+1}} \ar@{.>}[r]^-{\widehat{f}} & S(Y, k+1) \ar[d]^{\rho_{k} \circ \rho^c_{k+1} }\\
X \ar[r]_-{f} & Y.}$$
But then $F = \rho^c_{k+1}\circ \widehat{f}\colon S(X, 2k+1) \to S(Y, k)$ provides the desired covering of $f$.

Similarly, if  $k$ is odd, then use $F =\widehat{f}\circ  \rho^c_{k+1}\colon S(X, k+1) \to S(Y, k)$.
\end{proof}

\section{Lifting of Homotopies for Paths and Loops}\label{sec: homotopy}

Applications of the results of this paper will appear elsewhere.  But to indicate the way in which these fundamental results play a role in advancing our ``subdivision" agenda of developing homotopy theory in the digital setting, indicated in the Introduction, we include here one result.  We state a consequence of \thmref{thm: 2-D subdivision map rectangle} that, together with \corref{cor: loop subdivision map}, provides results  similar to path lifting and homotopy lifting results that play a prominent role in the development of  the fundamental group in the ordinary topological setting.  And, in fact, we rely on this result in \cite{LOS19c}, where we develop a digital fundamental group.

We use a  ``cylinder object" definition of homotopy, which is the one commonly used in  the digital topology literature.  In \cite{LOS19a} we give a fuller discussion of homotopy, including a ``path object" definition as well.

\begin{definition}\label{def: Homotopy}
Let $f, g\colon X \to Y$ be (continuous) maps of digital images.
We say that $f$ and $g$ are \emph{homotopic}, and write $f \approx g$,  if, for some $N\geq 1$, there is a continuous map
$$H \colon X \times I_N \to Y,$$
with $H(x, 0) = f(x)$ and $H(x, N) = g(x)$.  Then $H$ is a homotopy from $f$ to $g$.
\end{definition}

Suppose we have paths (of the same length)  in $Y$ with the same initial and terminal points.  That is, we have maps $\alpha, \beta \colon I_M \to Y$ with
$\alpha(0) = \beta(0) = y_0$ and $\alpha(M) = \beta(M)= y_M$ for some $y_0, y_M \in Y$.  If  $\alpha \approx \beta$, then the homotopy may be \emph{relative the endpoints}, which is to say that we have $H(0, t) = y_0$ and $H(M, t) = y_M$ for all $t \in I_N$.  If $\alpha$ and $\beta$ are loops in $Y$, so that $y_M = y_0$, and if $\alpha \approx \beta$ via a homotopy relative the endpoints, then we say that \emph{$\alpha$ and $\beta$ are homotopic via a based homotopy of based loops}.  The nomenclature comes from the setting of the fundamental group, as in \cite{LOS19c}, in which  $Y$ is a based digital image, and maps, loops, and homotopies are based.

The construction of $\widehat{H}$ in the proof of \thmref{thm: 2-D subdivision map rectangle} leads to the following ``covering homotopy" property of subdivisions.

\begin{corollary}[To \thmref{thm: 2-D subdivision map rectangle}]\label{cor: path homotopy covering}
Suppose $\alpha, \beta \colon I_M \to Y$ are paths in $Y$, with standard covers $\widehat{\alpha}, \widehat{\beta} \colon S(I_M, 2k+1) = I_{(2k+1)M+2k}\to S(Y, 2k+1)$ as in as in \thmref{thm: path odd subdivision map}.
\begin{itemize}
\item[(A)]   If $\alpha \approx \beta$, then  $\widehat{\alpha} \approx \widehat{\beta} \colon S(I_M, 2k+1) \to S(Y, 2k+1)$.
\item[(B)]   Suppose we have $\alpha(0) = \beta(0) = y_0$ and $\alpha(M) = \beta(M)= y_M$ for some $y_0, y_M \in Y$.  Then $\widehat{\alpha}(0) = \widehat{\beta}(0) = \overline{y_0}$ and $\widehat{\alpha}((2k+1)M+2k) = \widehat{\beta}((2k+1)M+2k) = \overline{y_M}$.
If $\alpha \approx \beta$ relative the endpoints,  then
$\widehat{\alpha} \approx \widehat{\beta}$ relative the endpoints.
\item[(C)]   Suppose we have $\alpha(0) = \alpha(M) = y_0 = \beta(0) = \beta(M)$, so that $\alpha$ and $\beta$ are loops based at some $y_0 \in Y$.
  Then $\widehat{\alpha}$ and $\widehat{\beta}$  are loops in $S(Y, 2k+1)$ (of length $(2k+1)M+2k$) based at $\overline{y_0}$.  If  $\alpha \approx \beta$ via a based homotopy of based loops,  then $\widehat{\alpha} \approx \widehat{\beta}$ via a based homotopy of based loops.
\end{itemize}
\end{corollary}

\begin{proof}
Part (A) is more-or-less a re-statement of the behaviour of $\overline{H}$ around the edges of the rectangle, from \thmref{thm: 2-D subdivision map rectangle}.  It follows from the construction of $\overline{H}$.  Then part (B) follows from the construction of $\overline{H}$ in \thmref{thm: 2-D subdivision map rectangle}, together with the fact that the standard cover of a constant path is a constant path (part (a) of \lemref{lem: technical stuff on covers}).  Part (C) is a special case of part (B).
\end{proof}

The ability to cover based homotopies of based loops in this way leads in \cite{LOS19c} to the result that our fundamental group constructed there is preserved by subdivision. That result is one of the major advances of \cite{LOS19c} over existing treatments of the fundamental group in the digital topology literature.  Other applications of the results of this paper appear in \cite{LOS19a}.

We believe that the results here for 1D and 2D domains may be extended for domains of any dimension.  However, in doing so there are many technical details to be resolved, as well as expositional challenges.   If it is possible to establish covering maps exist generally, for any dimension of domain, then it should be possible, for example, to develop higher homotopy groups in a way that incorporates subdivision similarly to the way in which  \cite{LOS19a} develops the fundamental group in a way that incorporates subdivision.


\providecommand{\bysame}{\leavevmode\hbox to3em{\hrulefill}\thinspace}
\providecommand{\MR}{\relax\ifhmode\unskip\space\fi MR }
\providecommand{\MRhref}[2]{%
  \href{http://www.ams.org/mathscinet-getitem?mr=#1}{#2}
}
\providecommand{\href}[2]{#2}

\end{document}